\documentclass[a4,12pt,usenames]{amsart}

\usepackage[utf8]{inputenc} 
\usepackage[T1]{fontenc}
\usepackage[english]{babel}
\usepackage{mathtools}
\usepackage{amssymb}
\usepackage{amsthm}
\usepackage{mathrsfs}
\usepackage{scalerel}
\def\mcirc{\mathbin{\scalerel*{\bigcirc}{t}}}

\usepackage{tikz,graphicx}

\usetikzlibrary{automata,positioning}
\usetikzlibrary{calc}

\usepackage{enumerate,url,float,lscape}
\usepackage{hyperref}

\usepackage[normalem]{ulem}

\newtheorem{lemma}{Lemma}[section]
\newtheorem{assumption}[lemma]{Assumption}

\newtheorem{example}[lemma]{Example}
\newtheorem{corollary}[lemma]{Corollary}
\newtheorem{definition}[lemma]{Definition}
\newtheorem{theorem}[lemma]{Theorem}
\newtheorem{remark}[lemma]{Remark}

\newtheorem{proposition}[lemma]{Proposition}

\numberwithin{equation}{section}
\numberwithin{lemma}{section}

\def\lambdaG{{\lambda_1(\Graph)}}
\def\lambdaGVD{{\lambda_1(\Graph;{{{\mVD}}})}}
 \def\mG{\mathsf{G}}

 \def\mV{\mathsf{V}}
   \def\mVD{\mathsf{V}_{\mathrm{D}}}
   \def\mVN{\mathsf{V}_{\mathrm{N}}}
   \def\mED{\mathsf{E}_{\mathrm{D}}}
   \def\mEN{\mathsf{E}_{\mathrm{N}}}
 \def\mE{\mathsf{E}}

 \def\torsd{T}
 \def\mv{\mathsf{v}}
 \def\me{\mathsf{e}}
 
 \def\mw{\mathsf{w}}

 \def\Grapht{\{\psi=t\}}
 \def\Gammat{\{\psi>t\}}
 
 \def\Gammas{\{\psi>s\}}
 \def\Graphtau{\{\psi=\tau\}}

\newcommand{\1}{\mathbf{1}}

\newcommand{\R}{\mathbb{R}}

\DeclareMathOperator{\real}{Re}
\DeclareMathOperator{\Dir}{Dir}
\newcommand{\Diri}{\#\Dir}
\newcommand{\C}{\mathbb{C}}
\newcommand{\N}{\mathbb{N}}

\newcommand{\ud}{\,\mathrm{d}}
\newcommand{\e}{\mathrm{e}}

\newcommand{\Graph}{\mathcal{G}} 
\newcommand{\NewGraph}{\widetilde{\Graph}} 

\DeclareMathOperator*{\essinf}{ess\;inf}

\DeclareMathOperator{\supp}{supp}
\newcommand{\Tmod}{T_{\mathrm{mod}}}

\DeclareMathOperator{\Inr}{Inr}

\DeclareMathOperator{\dist}{dist}

\title
[On torsional rigidity and ground-state energy of compact quantum graphs]
{On torsional rigidity and ground-state energy\\ of compact quantum graphs}

\author[D.~Mugnolo]{Delio Mugnolo}
\author[M.~Plümer]{Marvin Plümer}

\address{Delio Mugnolo, Lehrgebiet Analysis, Fakult\"at Mathematik und Informatik, Fern\-Universit\"at in Hagen, D-58084 Hagen, Germany}
\email{delio.mugnolo@fernuni-hagen.de}

\address{Marvin Plümer, Lehrgebiet Analysis, Fakult\"at Mathematik und Informatik, Fern\-Universit\"at in Hagen, D-58084 Hagen, Germany}
\email{marvin.pluemer@fernuni-hagen.de}



\thanks{The authors were partially supported by the Deutsche Forschungsgemeinschaft (Grant 397230547).
}

\begin{document}
\begin{abstract}
We develop the theory of torsional rigidity -- a quantity routinely considered for Dirichlet Laplacians on bounded planar domains -- for Laplacians on metric graphs {with at least one Dirichlet vertex}. Using a variational characterization that goes back to Pólya, we develop surgical principles that, in turn, allow us to prove isoperimetric-type inequalities: we can hence compare the torsional rigidity of general metric graphs with that of intervals of the same total length. In the spirit of the Kohler-Jobin Inequality, we also derive sharp bounds on the {ground-state energy} of a quantum graph in terms of its torsional rigidity: this is particularly attractive since computing the torsional rigidity reduces to inverting a matrix whose size is the number of the graph's vertices and is, thus, much easier than computing eigenvalues.
\end{abstract}

\maketitle

\tableofcontents

\section{Introduction}

Our aim in this article is to develop the theory of torsional rigidity for Laplacians on metric graphs: Following the terminology introduced by Pólya in~\cite{Pol48}, we will call \emph{torsion function} (with respect to the Dirichlet set ${{\mVD}}$) of a metric graph $\Graph$ with vertex set $\mV\supset {{\mVD}}\ne \emptyset$ the unique solution $v$ of the elliptic problem
\begin{equation}\label{eq:polya0}
\left\{
\begin{aligned}
-\Delta v(x)&=1,\quad &&x\in \Graph,\\
v(\mv)&=0,&&\mv\in {{\mVD}},
\end{aligned}
\right.
\end{equation}
and \emph{torsional rigidity} {its $L^1$-norm $T(\Graph):=T(\Graph;{{\mVD}}):=\|v\|_{L^1}$}. 

General self-adjoint differential operators supported on metric graphs are a popular toy model in both spectral theory and analysis of evolution equations~\cite{BerKuc13,Mug14,Kur22}: they usually go under the name of \emph{quantum graphs}. In particular, much attention has  been lately devoted to the study of the interplay of graphs' connectivity and metric properties on the one hand, and Laplacian eigenvalues on the other; while torsional properties of metric graphs have been, to the best of our knowledge, only discussed in~\cite{ColKagMcd16}.

 The theory of torsional rigidity of three-dimensional bodies with constant section $\Omega{\subset \R^2}$ goes back to Saint-Venant \cite{Sai55}, but the first rigorous mathematical contributions can be found in~\cite{Pol48,PolSze51}.
In his early meaning, torsional rigidity is a quantity in rational mechanics: defined as the $L^1$-norm of the solution $v$ of
\begin{equation}\label{eq:polya0-d}
\left\{
\begin{aligned}
-\Delta v(x)&=1,\quad &&x\in \Omega,\\
v(z)&=0,&&z\in \partial \Omega,
\end{aligned}
\right.
\end{equation}
 it is proportional (by $\theta\mu$, where $\mu$ is the body's {shear modulus}) to the couple resisting a given twist $\theta$. {Pólya observes that the torsional rigidity is ``a purely geometric constant, depending on size and shape of the domain''  \cite{Pol48}: indeed, a} classical result by Pólya -- clearly reminiscent of the Faber--Krahn inequality for the {ground-state energy} of the Dirichlet Laplacian and already conjectured by Saint-Venant based on physical considerations -- states that of all open bounded domains $\Omega\subset \R^2$ of given area $|\Omega|$, the circular one has the greatest torsional rigidity, i.e.,
\[
T(\Omega)\le T(\mcirc)= \frac{|\Omega|^2}{2\pi}.
\]
The crucial idea in Pólya's approach was the observation (\cite[page 272]{Pol48}, see also~\cite[\S~9B]{PolSze51}) that $T(\Omega)$ can be conveniently considered as a critical point of the Euler--Lagrange equation associated with~\eqref{eq:polya0-d} and that, accordingly, $T(\Omega)$ admits a variational characterization as
\begin{equation}\label{eq:polya-variat}
T(\Omega)=\sup_{v\in H^1_0(\Omega)} \frac{\left(\int_{\Omega} v\ud x\right)^2}{\|\nabla v\|^2_{L^2}}=\sup_{v\in H^1_0(\Omega)} \frac{\|v\|_{L^1}^2}{\|\nabla v\|^2_{L^2}},
\end{equation}
{see, e.g., \cite[Proposition~2.2]{Bra14}.}
Indeed, these suprema are attained and 
the unique maximizer of~\eqref{eq:polya-variat} is precisely the solution of~\eqref{eq:polya0-d}: it is called \emph{torsion function} in the literature (or sometimes \emph{warping function}, \cite[Section~II.2.3]{Ban80}). In the last decade it has attracted further interest after its role in wave 	localization phenomena for general Schrödinger operators has been greatly popularized by Filoche and Mayboroda~\cite{FilMay12}.
Ever since Pólya's pioneering study, investigations by Makai, Payne, Kohler-Jobin, and further authors have convincingly shown the rich theory of torsional rigidity in the context of shape optimization, often using symmetrization and rearrangement techniques.

This is the starting point of our article: Indeed, Pólya's basic definitions and properties of torsion on domains can be easily seen to carry over to metric graphs, up to replacing Lebesgue and Sobolev spaces on planar domains by their counterparts on metric graphs, thus leading to~\eqref{eq:polya0}. We are going to {consider the Laplacian $\Delta_{\Graph}$ on $\Graph$ with Dirichlet conditions on the vertices in ${{\mVD}}$ and} study two classes of problems: on one hand, we will prove sharp estimates on the torsional rigidity  $T(\Graph):={T(\Graph;{{\mVD}})}$ in its own right, like
\begin{equation}\label{eq:polya-m}
\frac{|\Graph|^3}{12|\mE|^2}\le T(\Graph{;{{\mVD}}})\le \frac{|\Graph|^3}{3}\qquad {\hbox{for all }\emptyset\ne {{\mVD}}\subset\mV}
\end{equation}
(see Equations~\eqref{eq:tors-first-1} and~\eqref{eq:firstlower-2}), 
where $|\Graph|$ is the total length of the metric graph $\Graph$ and $|\mE|$ is the number of its edges; { in particular, the torsional rigidity scales as $|\Graph|^3$}. {(In the case of domains, estimates like these are not available unless Dirichlet conditions are imposed \textit{on the whole boundary}. Because $\Delta_\Graph$ is not invertible in $L^2(\Graph)$ if ${{\mVD}}=\emptyset$, 
formally $T(\Graph;\emptyset)=\infty$: it is hence remarkable that the uniform estimate in~\eqref{eq:polya-m} holds as soon as ${{\mVD}}\ne\emptyset$.)}

On the other hand, by using techniques that are classical in the theory of of torsional rigidity, we will provide estimates on the {ground-state energy} 
\[
\lambdaGVD :=\min_{f\in H^1_0(\Graph;{{\mVD}})}\frac{\|f'\|^2_{L^2}}{\|f\|^2_{L^2}}
\]
 (i.e., the lowest eigenvalue of $-\Delta_\Graph$) by other objects, like
\begin{equation}\label{eq:heat-first-app}
\left(\frac{\pi}{\sqrt[3]{24 \|p_{\Graph{;{{\mVD}}}}\|_{L^1}}} \right)^2\le \lambdaGVD \le \frac{|\Graph|}{\|p_{\Graph{;{{\mVD}}}}\|_{L^1}}\qquad {\hbox{for all }\emptyset\ne {{\mVD}}\subset\mV}
\end{equation}
(see Equations~\eqref{eq:polya-upper} and~\eqref{eq:kj-vanilla}),
where $p_{\Graph{;{{\mVD}}}}\in L^1(\R_+\times \Graph\times \Graph)$ is the \emph{heat kernel}, i.e., the integral kernel of the heat semigroup generated by $\Delta_\Graph$ {with Dirichlet conditions on the vertices in ${{\mVD}}$}.

Not only will we replicate some of the most important results on torsion function and torsional rigidity of planar domains in the context of metric graph; we are also going to refine the direct counterparts of the classical results for special classes of metric graphs, like trees {(i.e., simply connected metric graphs)} or doubly connected graphs.
Indeed, in the one-dimensional setting of metric graphs techniques can be applied that seem to not be available in higher dimensional settings. On one hand, the equation~\eqref{eq:polya0} can be solved in a semi-explicit way
by solving an algebraic system of {$|\mV|-|{{\mVD}}|$ equations in $|\mV|-|{{\mVD}}|$ unknowns, where $|\mV|$ is the number of vertices of $\Graph$ and $|{{\mVD}}|$ is the number of Dirichlet vertices}: this paves the road to a geometric description of the torsion function that is much easier than for eigenfunctions. At the same time, the last decade has witnessed strong advances in the refinement of surgical principles for critical points of functionals defined on metric graphs: the recent article~\cite{BerKenKur19} is a comprehensive collection of techniques that go far beyond elementary test-function arguments.
We are thus going to borrow some ideas proposed by several authors for the study of spectral geometry of metric graphs, including~\cite{Fri05,Col15,KenKurMal16,BanLev17,BerKenKur17}, and combine them with techniques more typical of higher dimensional torsional theory \cite{Pol48,Koh78,Bra14}. In this way, not only can we reproduce in the metric graph context some well-known geometric bounds on the torsional rigidity and its product with (a power of) the {ground-state energy}; we can also sharpen some of them in the case of graphs of higher connectivity
-- a behavior that in the context of Laplacian eigenvalues of metric graphs has been discovered in~\cite{BanLev17,BerKenKur17}, but seems to have no counterpart in torsional theory of higher dimensional domains.

%

A further remarkable feature of torsional rigidity is its interplay with the heat equation, already hinted at in~\eqref{eq:heat-first-app}. Indeed, for the Laplacian $\Delta_\Graph$ on a metric graph $\Graph$ with Dirichlet conditions on a non-empty set ${{\mVD}}$ of vertices (and natural -- i.e., continuity and Kirchhoff -- conditions at all other vertices $\mV\setminus {{\mVD}}$), the quantity
	\[\mathcal Q_{\Graph}(t):=\|\e^{t\Delta_\Graph}\1\|_{L^1}=\int_\Graph\int_\Graph p_{\Graph{;{{\mVD}}}}(t;x,y)\ud y\ud x,\quad t> 0,\]
is called the \emph{heat content} of \(\Graph\) at time $t$: intuitively, the profile of $t\mapsto \mathcal Q_{\Graph}(t)$ describes how fast a metric graph is dissipating heat. 
Because $(\e^{t\Delta_{\Graph}})_{t\ge 0}$ is known to satisfy Gaussian estimates~\cite[Thm.~4.7]{Mug07},
\(
t\mapsto {\mathcal Q}_{\Graph}(t)
\)
is of class $L^1(0,\infty)$: we will call the $L^1$-norm of $Q_{\Graph}$, i.e.,
\begin{equation}\label{eq:intheacon}
\|{\mathcal Q}_\Graph\|_{L^1(0,\infty)}=\|p_{\Graph{;{{\mVD}}}}\|_{L^1}=\int_0^\infty \int_\Graph\int_\Graph p_{\Graph{;{{\mVD}}}}(t;x,y)\ud y\ud x\ud t,
\end{equation}
the \emph{integrated heat content} of $\Graph$ (with respect to the Dirichlet vertex ${{\mVD}}$). While it is known at latest since \cite{BerDav89} that the heat content -- at least in the case of domains in $\R^d$ -- carries interesting geometric information, estimates on ${\mathcal Q}_{\Graph}(t)$ are not easy to derive and will be discussed in a companion paper~\cite{MugPlu21b}. However, the \emph{integrated} heat content is a much more treatable quantity: our main results in this paper can { be} interpreted as geometric estimates (from above and below) of such integrated heat content, as in~\eqref{eq:heat-first-app}.
Indeed, the integrated heat content agrees with the torsional rigidity of $\Graph$, in view of elementary results from semigroup theory.

Let us present the plan of this article. 

After recalling some basic definitions in the theory of metric graphs, we introduce the torsional rigidity of a quantum graph in Section~\ref{sec:basic}. We also compute the torsional rigidity in a few simple examples.

In Section~\ref{sec:reduction} we present a first simple, yet effective tool for our analysis: we show (Proposition~\ref{prop:discrete-system}) that a quantum graph's torsion function can be computed explicitly -- unlike for eigenfunctions, this is true even in the non-equilateral case!  -- upon passing to a discretized problem: this boils down to solve an algebraic system of linear equations.

In Section~\ref{sec:surg} we introduce (Proposition~\ref{prop:surgery-0}) a new toolbox, mostly inspired by the spectral surgical methods developed in~\cite{KenKurMal16,BerKenKur17,BerKenKur19}. We use them to prove two main bounds on the torsional rigidity: a lower and an upper bound based on the inradius of $\Graph$ and on its total length, respectively (Proposition~\ref{prop:estim-low-inr} and Theorem~\ref{thm:polya}): both bounds can be improved if $\Graph$ is known to be {simply connected (i.e., $\Graph$ is a tree graph)} or doubly connected, respectively.

We then turn to our most significant topic: the derivation of estimates
on the {ground-state energy} of a quantum graph by means of its torsional rigidity. In Section~\ref{sec:spectralest} we  present an upper bound (Proposition~\ref{prop:polya-type}) whose domain counterpart goes back to Pólya (and which can be slightly modified to prove that the torsional rigidity also yields an upper estimate on the Cheeger constant of a metric graph); and a lower bound (Theorem~\ref{thm:kj-brasco-qg}), whose much more involved proof is  based on a rearrangement technique introduced by Kohler-Jobin and recently extended in~\cite{Bra14}. 

{ Finally, in Section~\ref{sec:torsland} we study a different but related topic: we discuss the possible use of the torsion function as a landscape function, elaborating on known results from \cite{BanCar01,GioSmi10,Ber12,FilMay12,Ste17} and extending them to the general setting of operators satisfying suitable forms of maximum principles on Banach lattices. Our observations about landscape functions are applied to metric graphs, but we offer an outlook to even more general settings.}

\begin{center}
\textsc{Acknowledgments}
\end{center}

\smallskip
We are grateful to Michiel van den Berg and Lorenzo Brasco
for stimulating discussions, useful comments and for pointing us at interesting references.

\medskip
\begin{center}
\textsc{Data availability statement}
\end{center}

\smallskip
All data generated or analysed during this study are included in this published article{.}

\section{Preliminaries and notation}\label{sec:basic}
Throughout this article let \(\Graph\) be a metric graph with edge set \(\mE=\mE_\Graph\) and vertex set \(\mV=\mV_\Graph\). We refer to~\cite{Mug19} for a precise introduction of the canonical structure of metric measure space induced on $\Graph$ by the Euclidean distance and the Lebesgue measure.

We impose the following assumption.

\begin{assumption}
\label{ass:graph}
The metric graph $\Graph$ is connected, i.e., there is a continuous path connecting any two points on the graph. It is compact and finite, i.e., it consists of finitely many edges of finite length.
\end{assumption}

 For an edge \(\me\in \mE\), let \(\ell_\me\) denote its length, and, for a vertex \(\mv\in \mV\), let \(\deg_\Graph(\mv)\) denote its degree, i.e.\ the number of edges incident in \(\mv\). Let \(\dist_\Graph:\Graph\times\Graph\rightarrow [0,\infty)\) denote the path metric on \(\Graph\). We suppose that \(\Graph\) has at least one vertex of degree \(1\). Let \({{\mVD}}=\mV_{{\mathrm{D}},\Graph}\) be a fixed subset of
	\[\{\mv\in\mV_\Graph~|~\deg_\Graph(\mv)=1\}\]
and let \({\mVN}:=\mV_{{\mathrm{N}},\Graph}=\mV\setminus {{\mVD}}\) be its complement in \(\mV\).  

Let \(\Delta_\Graph\) be the Laplacian on \(\Graph\) Dirichlet vertex condition in \({{\mVD}}\) and natural (i.e., the Kirchhoff-type condition
\begin{equation}\label{eq:kirchh}
\sum_{\me\in\mE_\mv}\frac{\partial u_\me}{\partial n}(\mv)=0,
\end{equation}
along with continuity across the vertices) vertex conditions on \({\mVN}\), i.e.\ its associated quadratic form \(a=a_{\Graph}\) is given by
	\[a_{\Graph}(u):=\int_{\Graph}|u'(x)|^2\mathrm{d}x=\sum_{\me\in E}\int_0^{\ell_\me}|u_\me'(x_\me)|^2\mathrm{d}x_\me\]
on the domain
	\[H^1_0(\Graph;{{\mVD}}):=\left\{u=(u_\me)_{\me\in \mE}\in\bigoplus_{\me\in \mE} H^1(0,\ell_\me)~\Big|~\begin{array}{l} u(\mv)=0\text{ for all }\mv\in {{\mVD}},\\ u\text{ is continuous in every }\mv\in {\mVN} \end{array}\right\}.\]
(Throughout this article, $u_\me$ denotes the restriction of $u$ to the edge $\me$.)
We refer to \({{\mVD}}\) as the set of \emph{Dirichlet vertices} of \(\Graph\) and to \({\mVN}\) as the set of \emph{natural vertices} of \(\Graph\).

\begin{definition}
Let $\Graph$ be a metric graph satisfying Assumptions~\ref{ass:graph} and let ${{\mVD}}$ be a non-empty subset of $\mV$. Then the \emph{torsion function} $v$ of $\Graph$ is the unique solution of
\begin{equation}\label{eq:polya0-bis}
\left\{
\begin{aligned}
-\Delta v(x)&=1,\quad &&x\in \Graph,\\
v(\mv)&=0,&&\mv\in {{\mVD}}.
\end{aligned}
\right.
\end{equation}
Additionally, we introduce the \emph{torsional rigidity}
\[
T(\Graph;{{\mVD}}):=\int_\Graph v(x)\ud x.
\]
\end{definition}

\begin{remark}
By definition, ${{\mVD}}$ is a set whose elements are all vertices of degree 1. At the same time, as far as the quadratic form $a_\Graph$ and hence the Laplacian spectrum are concerned, nothing changes if two or more Dirichlet vertices are glued, even though this certainly affects the topology of the underlying graph. To avoid confusion, we thus introduce the quantity
\[ \Diri(\Graph):=\sum\limits_{\mv\in{{\mVD}}}\deg(\mv),\]
which is invariant under gluing of Dirichlet vertices.
\end{remark}
 
At the danger or being redundant, we stress that 
the torsional rigidity $T(\Graph;{{\mVD}})$ does depend on the set ${{\mVD}}$ of vertices on which Dirichlet conditions are imposed; and that
the associated torsion function $v$ belongs to $D(\Delta_{\Graph})\subset C(\Graph)\cap \bigoplus_{\me\in \mE}C^1([0,\ell_\me])$ and satisfies the Dirichlet condition at the vertices in ${{\mVD}}$, hence in particular it belongs to $H^1_0(\Graph;{{\mVD}})$. If ${{\mVD}}$ is clear from the context, we will in the following write $H^1_0(\Graph)$ and $T(\Graph)$ instead of $H^1_0(\Graph;{{\mVD}})$ and $T(\Graph;{{\mVD}})$.

Pólya's proof of~\eqref{eq:polya-variat}  carries over verbatim to the case of compact metric graphs, leading to
\begin{equation}\label{eq:torsion-var}
T(\Graph)=
\sup_{u\in H^1_0(\Graph)} \frac{\left(\int_{\Graph} u\ud x\right)^2}{\|u'\|^2_{L^2}}=\max_{u\in H^1_0(\Graph)} \frac{\left(\int_{\Graph} u\ud x\right)^2}{\|u'\|^2_{L^2}}.
\end{equation}
in the following, we will refer to the term on the right hand side as \emph{Pólya quotient}.

 The maxima of this \emph{Pólya quotient}  in~\eqref{eq:torsion-var} form a one-dimensional space spanned by the torsion function $v$.

Since $\Graph$ is compact, and because $H^1_0(\Graph)$ is an ideal { of $H^1(\Graph)$ in the sense of Banach lattices, see \cite[Definition~2.19 and also Proposition~2.23]{Ouh05} (in particular, $u\in H^1_0(\Graph)$ implies $|u|\in H^1_0(\Graph)$)} and switching from $u$ to $|u|$ can only raise the quotient on the right hand side of \eqref{eq:torsion-var},  we find that 
\begin{equation}\label{eq:tors-graph-variat}
T(\Graph)=
\max_{u\in H^1_0(\Graph)} \frac{\left(\int_{\Graph} u\ud x\right)^2}{\|u'\|^2_{L^2}}=\max_{u\in H^1_0(\Graph)} \frac{\|u\|^2_{L^1}}{\|u'\|^2_{L^2}}.
\end{equation}
 
\begin{example}\label{exa:basic-tors}
In the 1-dimensional case the torsion function can be easily determined: if we let $a>0$ and consider the Poisson equation
\[
\left\{
\begin{aligned}
-v''(x)&=1,\qquad &&x\in (0,a),\\
-v'(0)&=\beta_0 v(0),\\
v'(a)&=\beta_1 v(a),
\end{aligned}
\right.
\]
with Robin boundary conditions, then the solution is given by
\begin{equation}\label{eq:robin}
v:(0,a)\ni x\mapsto 
-\frac{x^2}{2}
+\frac{\beta_0 a(a\beta_1-2)x}{2(a\beta_0 \beta_1-\beta_0-\beta_1)}
-\frac{a(a\beta_1-2)}{2(a\beta_0 \beta_1-\beta_0-\beta_1)}\in \R.
\end{equation}

(1) Let us now consider both relevant cases of a bounded interval of length $a$ with either mixed Dirichlet/Neumann conditions, or else with Dirichlet boundary conditions at both endpoints, by specializing~\eqref{eq:robin} to the cases of $\beta_0=\infty$ and, additionally,  either $\beta_1=\infty$ (both Dirichlet conditions) or $\beta_1=0$ (mixed Dirichlet/Neumann conditions).
 We will denote these configurations by $J_0$ and $J_1$, respectively, throughout this article.

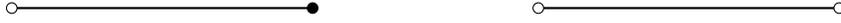
\begin{figure}[H]
\centering
\begin{tikzpicture}
	\coordinate (w0) at (0,0);
	\coordinate (w1) at (4,0);
		\draw[thick] (w0) -- (w1);
	\coordinate (w2) at (7,0);
	\coordinate (w3) at (11,0);
		\draw[thick] (w2) -- (w3);
		\draw[fill=white] (w0) circle (2pt);
		\draw[fill] (w1) circle (2pt);
		\draw[fill=white] (w2) circle (2pt);
		\draw[fill=white] (w3) circle (2pt);
\end{tikzpicture}
\medskip
\caption{The graph $J_0$ (left) and the graph $J_1$ (right); Dirichlet conditions are imposed at the white vertices.}\label{fig:J1}
\end{figure}

We immediately see that the torsion function on $J_0$ is
\[
v:(0,a)\ni x\mapsto \frac12 x\left(2a-x \right)\in \R
\]
and
\begin{equation}\label{eq:tors-1d-DN}
T(J_0)=\frac{a^3}{3};
\end{equation}
whereas the torsion function on $J_1$ is given by
\[
v:(0,a)\ni x\mapsto \frac{1}{2}x(a-x)\in \R.
\]
and 
\begin{equation}\label{eq:tors-1d}
T(J_1)=\frac{a^3}{12}.
\end{equation}

We observe that {the maximum of the torsion function on $J_0$ (resp., on $J_1$) is $\frac{a^2}{2}$ (resp., $\frac{a^2}{8}$); however,}
 the product of the lowest Laplacian eigenvalue $\lambda_1$ with the maximum of the torsion function {is {scale invariant} and} satisfies
\begin{equation}\label{eq:l1vmax}
\lambda_1 \|v\|_\infty=\frac{\pi^2}{8}.
\end{equation}

{
Also, let us observe that the quotient of the ground state $\varphi=\sin(\frac{\pi }{a}\cdot)$ and the torsion function $v$ satisfies on $J_1$ the estimate
\[
\frac{\varphi(x)}{v(x)}\le \frac{8}{a^2}\qquad\hbox{for all }x\in (0,a).
\]
}

(2) If $\Omega\subset \R$ is a bounded but disconnected set, then its torsional rigidity is the sum of its connected components' torsional rigidity:  the torsional rigidity of such a disconnected domain is thus strictly lower than that of the interval with same total length.  The same is clearly true of disconnected graphs.

(3) Let  $\mathcal S_k$ be an equilateral metric star with Dirichlet conditions imposed on all vertices but the central one (or, equivalently, a pumpkin graph with Dirichlet conditions on one vertex and natural conditions on the other one). 

\begin{figure}[H]
{
\centering
\begin{tikzpicture}
	\coordinate (w0) at (0:1);
	\coordinate (w1) at (72:1);
	\coordinate (w2) at (144:1);
	\coordinate (w3) at (216:1);
	\coordinate (w4) at (288:1);
	\foreach \i in {0,1,2,3,4} {
		\draw[thick] (0,0) -- (w\i);
		}
		\draw[fill] (0,0) circle (1.75pt);
		\foreach \i in {0,1,2,3,4} {
		\draw[fill=white] (w\i) circle (1.75pt);
		}
\end{tikzpicture}
}\\[10pt]
\caption{A star graph with 5 edges: Dirichlet conditions are imposed at the white vertices.}\label{fig:star-graph}
\end{figure}
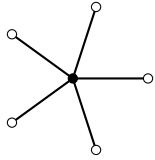

Let \(\me_1,\ldots,\me_k\) denote the edges of \(\mathcal S_k\). Each edge \(\me_j\) will be identified with the interval \([0,\frac{|\mathcal S_k|}{k}]\) where \(0\) corresponds to the degree \(1\) vertex and \(\frac{|\mathcal S_k|}{k}\) corresponds to the center vertex respectively. It is then immediate to check that the torsion function \(v\) on \(\mathcal S_k\) is given by
	\[v_{\me_j}(x_{\me_j})=\frac{1}{2}x_{\me_j}\left(\frac{2|\mathcal S_k|}{k}-x_{\me_j}\right)\]
for all \(x_{\me_j}\in [0,\frac{|\mathcal S_k|}{k}]\simeq \me_j\) and \(j=1,\ldots,k\). Thus, the torsional rigidity of \(\mathcal S_k\) is
\[
	T(\mathcal S_k)= \frac{k|\mathcal S_k|^3}{3k^3}= \frac{|\mathcal S_k|^3}{3k^2}.
\]

(4) Let $\Graph$ be a lasso graph consisting of a pendant edge $\me_1$ of length $\ell_1$ and a loop $\me_2$ of length $\ell_2$, see Figure~\ref{fig:lasso}, with Dirichlet conditions imposed on the vertex of degree 1.

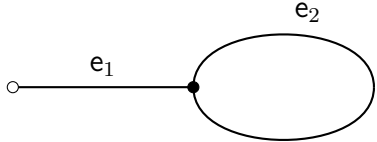
\begin{figure}[H]
\begin{tikzpicture}[scale=1.2]
\coordinate (a) at (0,0);
\coordinate (b) at (2,0);
\coordinate (c) at (4,0);
\draw[thick] (a) -- (b);
\draw[thick,bend left=90]  (b) edge (c);
\draw[thick,bend right=90]  (b) edge (c);
\draw[fill=white] (0,0) circle (1.75pt);
\draw[fill] (2,0) circle (1.75pt);
\node at (1,0) [anchor=south] {$\me_1$};
\node at (3,0.6) [anchor=south west] {$\me_2$};
\end{tikzpicture}
\caption{The lasso graph}\label{fig:lasso}
\end{figure}
The torsion function $v$ is necessarily given by
\[
\begin{split}
v_{\me_1}(x_{\me_1})&=-\frac{x^2_{\me_1}}{2}+a_1x_{\me_1}+b_1,\qquad x_{\me_1}\in [0,\ell_1],\\
v_{\me_2}(x_{\me_2})&=-\frac{x^2_{\me_2}}{2}+a_2x_{\me_2}+b_2,\qquad x_{\me_2}\in [0,\ell_2]:
\end{split}
\]
imposing the boundary conditions allows us to determine the coefficients $a_1,a_2,b_1,b_2$, thus concluding that the torsion function is
\[
\begin{split}
v_{\me_1}(x_{\me_1})&=-\frac{x^2_{\me_1}}{2}+(\ell_1+\ell_2)x_{\me_1},\qquad x_{\me_1}\in [0,\ell_1],\\
v_{\me_2}(x_{\me_2})&=-\frac{x^2_{\me_2}}{2}+\frac{\ell_2}{2}x_{\me_2}+\frac{\ell_1}{2}\left(\ell_1+2\ell_2\right),\qquad x_{\me_2}\in [0,\ell_2]:
\end{split}
\]
unsurprisingly, the maximum
\[
\|v\|_\infty=v_{\me_2}\left(\frac{\ell_2}{2}\right)=\frac{\ell_1^2}{2}+\ell_1\ell_2 +\frac{\ell_2^2}{8}
\]
 is attained at the midpoint of $\me_2$, i.e., at the point of maximal distance from the Dirichlet vertex.
A direct computation shows that the torsional rigidity is
	\begin{equation}\label{eq:lasso-tors}
		T(\Graph) = \frac{1}{12}(\ell_1^3+\ell_2^3)+\frac{1}{4}\left(\ell_1+2\ell_2\right)^2\ell_1.
	\end{equation}
This procedure will be generalized in the next Section.
\end{example}

 \section{Reduction to a discrete problem}\label{sec:reduction}

Let us now show how the computation of a metric graph's torsion can be reduced to the computation of the solution of an algebraic system of $|\mV|-|{{\mVD}}|$ equations, {as sketched already in~\cite[Section~4]{ColKagMcd16}.} { In fact, we will see that the metric torsion and torsional rigidity can be rewritten via discrete counterparts assigned to a corresponding weighted combinatorial graph, see Definition~\ref{def:discrete-tors} below.}

\begin{proposition}\label{prop:discrete-system}
	Let \(v\) be the torsion function of \(\Graph\). Then the restriction \(g:=v_{|\mV}:\mV\rightarrow \mathbb R\) is the unique solution of the system
	\begin{equation}\label{eq:discr-tors}
	\left\{\begin{aligned}
		g(\mv) & = 0, \quad &\mv\in{{\mVD}},\\[.1cm]
		\displaystyle\frac{1}{d_\mv^\ell}\sum_{\me=\{\mv,\mw\}\in \mE_\mv} \frac{g(\mv)-g(\mw)}{\ell_\me} & =\displaystyle\frac{1}{2}, & \mv\in{\mVN},
	\end{aligned}\right.
	\end{equation}
	where \(d_\mv^\ell:=\sum_{\me\in\mE_\mv}\ell_\me\) denotes the metric degree of the vertex \(\mv\).
\end{proposition}

{
In view of~\eqref{eq:torsion-edgewise} below, Proposition~\ref{prop:discrete-system} shows that solving~\eqref{eq:discr-tors} immediately delivers a closed-form expression for the torsion function. Albeit elementary, this observation seems to be novel: for example, the authors of~\cite{HarMal20} regret that

\begin{quote}
\emph{
[...] except in rare cases of high symmetry, the torsion function [...] [is] only
computationally known.}
\end{quote}
See Example~\ref{exa:coll-tors-comput} for a class of (rather asymmetrical) graphs whose (discrete and, hence, also continuous) torsion functions can be  easily computed.
}

\begin{proof}
	Since the torsion function \(v\) solves the equation \(-v_\me''=1\) edgewise, it satisfies
	\begin{equation}\label{eq:torsion-edgewise}	
		v_\me(x_\me)=\frac{1}{2}(\ell_\me-x_\me)x_\me+\frac{v_\me(\ell_\me)-v_\me(0)}{\ell_\me}x_\me+v_\me(0)\qquad\hbox{for all }x_\me\in [0,\ell_\me].
	\end{equation}
Then
\begin{equation}\label{eq:torsion-edgewise-deriv}	
v_\me'(x_\me)=\frac{\ell_\me-2x_\me}{2}+\frac{v_\me(\ell_\me)-v_\me(0)}{\ell_\me},\qquad x_\me\in [0,\ell_\me],
\end{equation}
	holds. Therefore, the natural vertex conditions imposed on the elements  \(\mv\in\mV\setminus {{\mVD}}\) yield for the restriction of $v$ to ${\mVN}$
	\begin{align*}
		0 & = \sum_{\me=\{\mv,\mw\}\in\mE_\mv}-\frac{\ell_\me}{2}+\frac{f(\mv)-f(\mw)}{\ell_\me}\\
		& = -\frac{d_\mv^\ell}{2} + \sum_{\me=\{\mv,\mw\}\in \mE_\mv} \frac{f(\mv)-f(\mw)}{\ell_\me}
	\end{align*}
	This concludes the proof.
\end{proof}

We can interpret the left-hand-side of the system considered in Proposition \ref{prop:discrete-system} as a self-adjoint operator in a weighted \(\ell^2\)-space. First, let \({\mEN}\) denote the set of edges in \(\mE\) whose  incident vertices \(\mv,\mw\) are both elements of ${\mVN}$, and let ${\mED}$ denote the set of edges where at least one incident vertex is a Dirichlet vertex, i.e., \({\mED}:=\mE\setminus {\mEN}\). We consider vertex and edge weights \(m, \mu\) defined as follows: 
	\begin{align*}
		m(\mv) =d_\mv^\ell,\quad\mv\in{\mVN},\qquad\hbox{and}\qquad \mu(\me) =\frac{1}{\ell_\me},\quad \me\in{\mEN}.
	\end{align*}
These weights have extensively discussed in \cite{KosMalNei17,ExnKosMal18,Plu21}, in the context of spectral theory of quantum graphs. {We refer to $\mG=\mG_{m,\mu}$ constructed above as the \emph{weighted combinatorial graph underlying} the metric graph $\Graph$.}
	
	Furthermore, we consider the non-negative discrete potential \(c:{\mVN}\rightarrow [0,\infty)\) defined by
	\begin{align*}
		c(\mv) = \sum_{\me\in \mE_\mv\cap{\mED}}\frac{1}{\ell_\me},\quad \mv\in{\mVN}.
	\end{align*}
	On the weighted Hilbert space  \(\ell^2_m({\mVN})\) of square summable functions \(f:{\mVN}\rightarrow \mathbb C\) whose inner product is given by
		\[(f,g)_{\ell^2_m({\mVN})}:=\sum\limits_{\mv\in{\mVN}}m(\mv)\overline{f(\mv)}g(\mv),\qquad f,g:{\mVN}\to \C,
		\]
we consider the quadratic form \(q=q_{m,\mu,c}\) given by
		\[q(f)=\sum_{\me=\{\mv,\mw\}\in{\mEN}}\mu(\me)|f(\mv)-f(\mw)|^2+\sum\limits_{\mv\in{\mVN}}c(\mv)|f(\mv)|^2,\quad f\in\ell^2_m({\mVN}).\]
	The positive definite (in particular, invertible), self-adjoint operator \(\mathcal L\) on $\ell_m^2({\mVN})$ associated with \(q\)  is given by
	\begin{equation}\label{eq:discrete-operator}
		(\mathcal L f)(\mv)=\frac{1}{m(\mv)}\left[\sum_{\me=\{\mv,\mw\}\in\mE_\mv\cap{\mEN}}\mu(\me)(f(\mv)-f(\mw))+c(\mv)f(\mv)\right],\quad f\in\ell_m^2({\mVN}),\ \mv\in{\mVN}.
	\end{equation}
The following result is an immediate consequence of Proposition \ref{prop:discrete-system}.
	\begin{corollary}\label{cor:discrete-torsion}
		If \(v\) is the torsion function of \(\Graph\), then \(v_{|{\mVN}}=\frac{1}{2}\mathcal L^{-1}\1_{{\mVN}}\)
		 holds.
	\end{corollary}
	{
	\begin{definition}\label{def:discrete-tors}
Given a metric graph $\Graph$, we will refer to \(g:=\mathcal L^{-1}\1_{{\mVN}}\) as the \emph{discrete torsion function} of the underlying weighted combinatorial graph \(\mG\). Its \(\ell_m^1\)-norm will be called the \emph{discrete torsional rigidity}
	\begin{equation}\label{eq:tors-discrtors}
		\torsd(\mG):=\sum\limits_{\mv\in\mV}m(\mv)g(\mv)=\left\|\mathcal L^{-1}\1_{{\mVN}}\right\|_{\ell_m^1({\mVN})}.
	\end{equation}
	\end{definition}
}
As in the continuous case, it is straightforward to derive the variational characterization
\begin{equation}
	\torsd(\mG) = 
	\sup_{f\in \ell^2_m({\mVN})}\frac{\left(\sum\limits_{\mv\in{\mVN}}m(\mv)f(\mv)\right)^2}{q(f)}=
	\max_{f\in \ell^2_m({\mVN})}\frac{\left(\sum\limits_{\mv\in{\mVN}}m(\mv)|f(\mv)|\right)^2}{q(f)}
\end{equation}
where the discrete torsion function is -- up to rescaling -- the unique maximizer of the functional appearing on the right-hand-side.

\begin{lemma}\label{lem:positivity-of-discrete-torsion}
	The discrete torsion function \(g=\mathcal L^{-1}\1_{{\mVN}}\) is strictly positive on \({\mVN}\).
\end{lemma}
\begin{proof}
	Let \(\mv_{\min}\) be a vertex in \({\mVN}\) with \(g(\mv_{\min})\leq g(\mv)\) for all \(\mv\in{\mVN}\). Then, by definition of \(g\), we have
		\[c(\mv_{\min})g(\mv_{\min})=m(\mv_{\min})+\sum_{\me=\{\mv,\mw\}\in\mE_\mv\cap{\mEN}}\mu(\me)(g(\mw)-g(\mv_{\min}))>0.\]
	Since \(c\) is non-negative on \({\mVN}\), we obtain \(g(\mv_{\min})>0\) and thus \(g\) is strictly positive on \({\mVN}\).
\end{proof}
As an immediate consequence of \eqref{eq:torsion-edgewise} and Lemma \ref{lem:positivity-of-discrete-torsion} we obtain:
\begin{corollary}\label{cor:positivity-of-torsion}
	The torsion function \(v\) on \(\mathcal G\) is strictly positive on \(\mathcal G\setminus{{\mVD}}\).
\end{corollary}

Let us now collect a few properties of the torsion function, for the sake of future reference.

\begin{lemma}\label{lem:no-local-min}
The torsion function is a Lipschitz continuous over the metric space $\Graph$. It is edgewise strictly concave and has no local minima outside ${{\mVD}}$.\end{lemma}
\begin{proof}
Because the torsion function satisfies edgewise $-v_\me''=\1$, and because it is twice (in fact, infinitely often) differentiable in the interior of each edge, we conclude from~\eqref{eq:torsion-edgewise}	 that
it is edgewise a strictly concave function;
%
any critical point of $v$ in the interior of an edge would then necessarily be a \emph{strict} local maximum, rather than a minimum. 

Would a local minimum of $v$ occur instead at a vertex $\mv\in \mV$, then $v$ would be strictly monotonically increasing, in a neighborhood of $\mv$, along all outgoing edges: hence the normal derivatives at $\mv$ would all be strictly positive: a contradiction to the fact that $v$ lies in the domain of $\Delta_{\Graph}$ and, thus, satisfies the Kirchhoff-type condition.

Finally, it follows from Lemma~\ref{lem:lip} below that 
$v$ is Lipschitz continuous.
\end{proof}

{In order to complete the proof of Lemma~\ref{lem:no-local-min}, we need the following fact: it is probably already known, but we could not find it in the literature.
\begin{lemma}\label{lem:lip}
Under the Assumptions~\ref{ass:graph}, the space
\[
C(\Graph)\cap \bigoplus_{\me\in\mE}H^2(0,\ell_\me)
\]
is continuously embedded in the space $\mathrm{Lip}(\Graph)$ of Lipschitz continuous functions over the metric space $\Graph$.
\end{lemma}

\begin{proof}
Let $u\in W^{1,\infty}(\Graph)=C(\Graph)\cap \bigoplus_{\me\in\mE} W^{1,\infty}(0,\ell_\me)$. Let \(x,y\in\Graph\) and let \(\gamma\) be a path in \(\Graph\) connecting \(x\) and \(y\). We then have
	\[|u(x)-u(y)|=\left|\int_\gamma u'(t) \ud t\right|\leq L(\gamma)\|u'\|_\infty\]
where \(L(\gamma)\) denotes the length of \(\gamma\). Since \(\gamma\) is arbitrary, we obtain
	\[|u(x)-u(y)|\leq \|u'\|_\infty d_\Graph(x,y).\]
Therefore, \(W^{1,\infty}(\Graph)\) is continuously embedded in \(\mathrm{Lip}(\Graph)\). This however already yields the assertion, because the embedding $H^2\hookrightarrow W^{1,\infty}$ is continuous edgewise.
\end{proof}
\begin{remark}\label{rem:geo-infin}
(1) Trivially, Lemma~\ref{lem:lip} also implies that all eigenfunctions of self-adjoint realizations of metric graph Laplacians (or, in fact, Schrödinger operators) with continuity and/or Dirichlet conditions at all vertices are Lipschitz continuous functions. Remarkably, this is also true of the heat kernel of $(\e^{t\Delta_{\Graph}})_{t\ge 0}$ with respect to each of its two spacial variables. (Observe that \emph{joint} Hölder continuity (with exponent $\frac12$) of the heat kernel has been recently proved in~\cite[Theorem~5.2]{KosMugNic21}.)

(2) 
Observe that Lemma~\ref{lem:lip} also extends to infinite metric graphs, as the proof merely uses that \(\Graph\) is a length space.
\end{remark}
}

The following result bridges the distance between the torsional rigidity of metric and underlying weighted combinatorial graphs.

\begin{theorem}\label{thm:discrete-and-metric-torsion}
	The discrete and metric torsional rigidity are related via the formula
	\begin{equation}
		T(\Graph)=\sum_{\me\in\mE}\frac{\ell_\me^3}{12}+\frac{1}{4}\torsd(\mG).
	\end{equation}
\end{theorem}


\begin{proof}
	Let \(v\) be the torsion function of \(\Graph\) and let \(g\) be the discrete torsion function of \(\mG\). 
By Corollary \ref{cor:discrete-torsion} we have \(v_{|{\mVN}}=\frac{1}{2}g\). Then, using \eqref{eq:torsion-edgewise} we obtain
	\begin{align*}
		T(\Graph) & =\sum_{\me\in\mE}\int_0^{\ell_\me}v_\me(x_\me)\ud x_\me\\
		& =\sum_{\me\in\mE}\int_0^{\ell_\me}\frac{1}{2}(\ell_\me-x_\me)x_\me\ud x_\me+\sum_{\me\in\mE}\int_0^{\ell_\me} \left(\frac{v_\me(\ell_\me)-v_\me(0)}{\ell_\me}x_\me+v_\me(0)\right)\ud x_\me,
	\end{align*}
with
	\begin{equation*}
		\int_0^{\ell_\me}\frac{1}{2}(\ell_\me-x_\me)x_\me\ud x_\me = \frac{\ell_\me^3}{12},
	\end{equation*}
cf.~\eqref{eq:tors-1d}, and
	\begin{align*}
	\sum_{\me\in\mE}\int_0^{\ell_\me} \left(\frac{v_\me(\ell_\me)-v_\me(0)}{\ell_\me}x_\me+v_\me(0)\right)\ud x_\me
		 & = \frac{1}{2}\sum_{\me\in\mE}\ell_\me(v_\me(\ell_\me)+v_\me(0))\\
		 & = \frac{1}{2}\sum\limits_{\mv\in\mV}d_\mv^\ell v(\mv)\\
		 & = \frac{1}{4}\sum\limits_{\mv\in{\mVN}}m(\mv)g(\mv)\\
		 & = \frac{1}{4}\torsd(\mG).
	\end{align*}
This completes the proof.
\end{proof}

\begin{example}\label{exa:coll-tors-comput}
Let \(\Graph\) be a \emph{stower graph}, i.e.\ the edge set \(\mE\) consists of a finite number of loops (the \emph{petals}) and leaves, all of them connected to a single center vertex \(\mv_c\) -- see Figure \ref{fig:stower-graph}. 

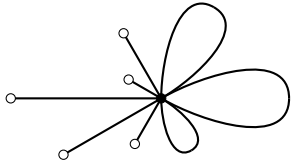
\begin{figure}[H]
\centering
\begin{tikzpicture}
	\coordinate (w0) at (120+30*0:1);
	\coordinate (w1) at (120+30*1:0.5);
	\coordinate (w2) at (120+30*2:2);
	\coordinate (w3) at (120+30*3:1.5);
	\coordinate (w4) at (120+30*4:0.7);
	\foreach \i in {0,1,2,3,4} {
		\draw[thick] (0,0) -- (w\i);
		}
	\coordinate (v1) at (60*1:1.4);
	\coordinate (v5) at (60*5:0.8);
	\coordinate (v6) at (60*6:1.7);
	\foreach \i in {1,5,6} {
		\draw[thick] (0,0) to [out=60*\i-30,in=60*\i+270] (v\i);
		\draw[thick] (0,0) to [out=60*\i+30,in=60*\i+90] (v\i);
	}
		\draw[fill] (0,0) circle (1.75pt);
		\foreach \i in {0,1,2,3,4} {
		\draw[fill=white] (w\i) circle (1.75pt);
		}
\end{tikzpicture}
\caption{A stower graph with \(|\mE_p|=3\) petals and \(|\mE_l|=5\) leaves: Dirichlet conditions are imposed at the white vertices.}\label{fig:stower-graph}
\end{figure}

Let \(\mE_p\) denote the set of petals and \(\mE_l\) denote the set of leaves. We impose Dirichlet conditions at all of the leaf vertices; in particular \({\mVN}=\{\mv_c\}\) and, in our notation, $\mE_l={\mED}$ and $\mE_p={\mEN}$. The vertex weight \(m\) and discrete potential \(c\) associated with \(\Graph\) are given by
	\begin{align*}
		m(\mv_c) & = \sum_{\me\in \mE_l}\ell_\me+2\sum_{\me\in \mE_p}\ell_\me, &
		c(\mv_c) & = \sum_{\me\in \mE_l}\frac{1}{\ell_\me},
	\end{align*}
	and thus the discrete torsion function \(g\) and the discrete torsional rigidity are given by
	\begin{align*}
		g(\mv_c) & = \left(\sum_{\me\in \mE_l}\ell_\me+2\sum_{\me\in \mE_p}\ell_\me\right)\cdot \left(\sum_{\me\in \mE_l}\frac{1}{\ell_\me}\right)^{-1}
	\end{align*}
	and
	\begin{align*}
		\torsd(\mG) & = m(\mv_c)g(\mv_c)= \left(\sum_{\me\in \mE_l}\ell_\me+2\sum_{\me\in \mE_p}\ell_\me\right)^2\cdot \left(\sum_{\me\in \mE_l}\frac{1}{\ell_\me}\right)^{-1}.
	\end{align*}
	Therefore, by Theorem \ref{thm:discrete-and-metric-torsion}, the torsional rigidity of the stower graph is
	\begin{equation}\label{eq:stower-tors}
		T(\Graph) = \frac{1}{12}\sum_{\me\in\mE}\ell_\me^3+\frac{1}{4}\left(\sum_{\me\in \mE_l}\ell_\me+2\sum_{\me\in \mE_p}\ell_\me\right)^2\cdot \left(\sum_{\me\in \mE_l}\frac{1}{\ell_\me}\right)^{-1};
	\end{equation}
in particular, in the equilateral case we find
	\begin{equation}\label{eq:stower-tors-equi}
		T(\Graph) 
=\frac{|\Graph|^3}{4|\mE|^3}\left(\frac{|\mE|}{3}+\frac{(|\mE_l|+|\mE_p|)^2}{|\mE_l|}\right).
	\end{equation}
(Observe that we recover Example~\ref{exa:basic-tors}.(3) in the case of $\mE_p=\emptyset$.)

Because $g(\mv_c)$ is explicitly known, the torsion function $v$ can be written down explicitly, much like in the case of the lasso graph in Example~\ref{exa:basic-tors}.(4).

\end{example}
\subsection{A Hadamard-type formula
}
As an application of the above results on the discrete torsion function, let us conclude this section by 
studying the behavior of the torsion {function} and torsional rigidity under perturbation of the graph's edge lengths. We fix one edge \(\me_0\in\mE\). For \(s>0\), we consider the perturbed graph \(\Graph_s\) with the same topology as \(\Graph\) and whose edge lengths \((\ell_{s,\me})_{\me\in\mE}\) are given by
\[
\ell_{s,\me} = \begin{cases}
	s, & \me=\me_0,\\
	\ell_\me, & \me\neq\me_0,
\end{cases}
\]
i.e., the lengths of the edges \(\me\neq \me_0\) are fixed, while the lengths of the edge \(\me_0\) is variable, and for \(s=\ell_{\me_0}\) we obtain our original graph \(\Graph\). For \(s>0\), the torsion on \(\Graph_s\) will be denoted by \(v_s=(v_{s,\me})_{\me\in\mE}\).

Analogously to known Hadamard-type formulas for eigenvalues for Laplacians on metric graphs (see \cite[Proof of the Lemma]{Fri05}, \cite[Appendix A]{Col15}, and \cite[Lemma 5.2]{BanLev17}), we can derive a Hadamard-type formula for the torsion{al rigidity} of metric graphs:
\begin{proposition}\label{prop:hada}
	The map \((0,\infty)\ni s\mapsto T(\Graph_s)\) is continuously differentiable and its derivative in \(s=\ell_{\me_0}\) is given by
	\begin{equation}\label{eq:hadamard-torsion}
		\frac{\mathrm{d}}{\mathrm{d} s}_{|s=\ell_{\me_0}}T(\Graph_s)=v'(x)^2+2v(x)>0,
	\end{equation}
	where \(v\) is the torsion on \(\Graph\) and \(x\) is an arbitrary element of \(\me_0\). (In particular, the right-hand-side in \eqref{eq:hadamard-torsion} does not depend on the particular choice of \(x\in\me_0\).)
\end{proposition}
\begin{proof}
	We first observe that the right-hand-side in \eqref{eq:hadamard-torsion} is independent of \(x\in \me_0\simeq [0,\me_0]\), since
	\[
		\frac{\mathrm{d}}{\mathrm{d} x}(v'(x)^2+2v(x))=2v''(x)v'(x)+2v'(x)=0,
	\]
	where we used that the torsion satisfies the equation \(-v''={\mathbf 1}\) edgewise.
	
	Next, we establish that \((0,\infty)\ni s\mapsto T(\Graph_s)\) is indeed continuously differentiable. For that purpose, let \(\mathcal L_s\) be the discrete operator defined via \eqref{eq:discrete-operator} associated with the metric graph \(\Graph_s\) and let \(g_s=\mathcal L_s^{-1}\mathbf 1_{{\mVN}}\) be the associated discrete torsion. Since the entries of \(\mathcal L_s\) with respect to the canonical basis of \(\mathbb R^{{\mVN}}\) are continuously differentiable functions of \(s\) it follows from Cramer's rule that the \(g_s(\mv)\) is a continuously differentiable function of \(s\). Therefore \eqref{eq:tors-discrtors} yields that \((0,\infty)\ni s\mapsto T(\Graph_s)\) is continuously differentiable. We obtain even more: by Corollary \ref{cor:discrete-torsion} we have \(v_{s|{\mVN}}=g_s\). It thus follows from \eqref{eq:torsion-edgewise} and \eqref{eq:torsion-edgewise-deriv} that for each edge \(\me\in\mE\) and each \(x_\me\in\me\simeq [0,\ell_{s,\me}]\) both \(v_{s,\me}(x_\me)\) and \(v_{s,\me}'(x_\me)\) are continuously differentiable functions of \(s\).
	
	Our next aim is to derive \eqref{eq:hadamard-torsion}. Note that it suffices to derive it for one point on \(\me_0\). We first recall that the torsion \(v_s\) satisfies
	\begin{equation}\label{eq:s-dependent-torsion}
		T(\Graph_s) = \int_{\Graph_s} v_s(x)\,\mathrm{d} x = \int_{\Graph_s} v_s'(x)^2\,\mathrm{d} x.
	\end{equation}
	Using the Leibniz rule we find
	\begin{equation}\label{eq:derivative-hadamard-l1}
		\begin{split}
		\frac{\mathrm{d}}{\mathrm{d} s}\int_{\Graph_s} v_s(x)\,\mathrm{d} x
		& =
		\sum_{\me\in \mE}\frac{\mathrm{d}}{\mathrm{d} s}\int_0^{\ell_{s,\me}} v_{s,\me}(x)\,\mathrm{d} x_\me\\
		& =
		v_{s,\me_0}(s)+\sum_{\me\in \mE}\int_0^{\ell_{s,\me}} \frac{\mathrm{d}}{\mathrm{d} s}v_{s,\me}(x)\,\mathrm{d} x_\me\\
		& =
		v_{s,\me_0}(s)+\int_{\Graph_s} \frac{\mathrm{d}}{\mathrm{d} s}v_s(x)\,\mathrm{d} x
		\end{split}
	\end{equation}
	and
	\begin{equation*}
	\begin{split}
 \frac{\mathrm{d}}{\mathrm{d} s}\int_{\Graph_s} |v_s'(x)|^2\,\mathrm{d} x= &
		\sum_{\me\in \mE}\frac{\mathrm{d}}{\mathrm{d} s}\int_0^{\ell_{s,\me}} v_{s,\me}'(x_\me)^2\,\mathrm{d} x_\me\\
		= &
		v_{s,\me_0}'(s)^2+2\sum_{\me\in \mE}\int_0^{\ell_{s,\me}} v_{s,\me}'(x_\me)\frac{\mathrm{d}}{\mathrm{d} s}v_{s,\me}'(x_\me)\,\mathrm{d} x_\me\\
		= &
		v_{s,\me_0}'(s)^2-2\sum_{\me\in \mE}\int_0^{\ell_{s,\me}} v_{s,\me}''(x_\me)\frac{\mathrm{d}}{\mathrm{d} s}v_{s,\me}(x_\me)\,\mathrm{d} x_\me\\
		&\quad+2\sum_{\me\in \mE}v_{s,\me}'(\ell_{s,\me})\frac{\mathrm{d}}{\mathrm{d} s}v_{s,\me}(\ell_{s,\me})-v_{s,\me}'(0)\frac{\mathrm{d}}{\mathrm{d} s}v_{s,\me}(0)\\
		= &
		v_{s,\me_0}'(s)^2+2\int_{\Graph_s} \frac{\mathrm{d}}{\mathrm{d} s}v_{s}(x)\,\mathrm{d} x+2\sum_{\me\in \mE}v_{s,\me}'(\ell_{s,\me})\frac{\mathrm{d}}{\mathrm{d} s}v_{s,\me}(\ell_{s,\me})-v_{s,\me}'(0)\frac{\mathrm{d}}{\mathrm{d} s}v_{s,\me}(0)
	\end{split}
	\end{equation*}
	where we used \(-\Delta_{\Graph_s}v_{s}=\mathbf 1_{\Graph_s}\) in the last step. Now, for \(\me\in\mE\), we have
	\[
		\frac{\mathrm{d}}{\mathrm{d} s}\left(v_{s,\me}(\ell_{s,\me})\right)=\begin{cases}
			\frac{\mathrm{d}}{\mathrm{d} s}v_{s,\me}(\ell_{\me}), & \me\neq\me_0,\\
			\frac{\mathrm{d}}{\mathrm{d} s}v_{s,\me_0}(s)+v'_{s,\me_0}(s), & \me=\me_0.
		\end{cases}
	\]
	Using the boundary conditions imposed at the vertices of \(\mathcal G_s\) we obtain
	\begin{align*}
	& 		\sum_{\me\in \mE}v_{s,\me}'(\ell_{s,\me})\frac{\mathrm{d}}{\mathrm{d} s}v_{s,\me}(\ell_{s,\me})-v_{s,\me}'(0)\frac{\mathrm{d}}{\mathrm{d} s}v_{s,\me}(0)\\
	= &
	-v'_{s,\me_0}(s)^2-\sum\limits_{\mv\in{\mVN}}\frac{\mathrm{d}}{\mathrm{d} s}\left(v_{s}(\mv)\right)\sum_{\me\in\mE_\mv} \frac{\partial v_{s,\me}}{\partial n}(\mv)\\
	& = -v'_{s,\me_0}(s)^2.
	\end{align*}
	We infer that
	\begin{equation}\label{eq:derivative-hadamard-form}
		\frac{\mathrm{d}}{\mathrm{d} s}\int_{\Graph_s} |v_s'(x)|^2\,\mathrm{d} x = -v_{s,\me_0}'(s)^2+2\int_{\Graph_s} \frac{\mathrm{d}}{\mathrm{d} s}v_{s}(x)\,\mathrm{d} x.
	\end{equation}
	Using \eqref{eq:s-dependent-torsion}, \eqref{eq:derivative-hadamard-l1} and \eqref{eq:derivative-hadamard-form}, we finally obtain
	\begin{align*}
		\frac{\ud}{\ud s}T(\Graph_s) & = \frac{\ud}{\ud s}\left(2\int_{\Graph_s} v_s(x)\,\ud x - \int_{\Graph_s} v_s'(x)^2\,\ud x\right)\\
		& = 2v_{s,\me_0}(s)+v_{s,\me_0}'(s)^2
	\end{align*}
	which proves the asserted equation.
\end{proof}


\section{Torsional surgery and inverse problems}\label{sec:surg}

The variational characterization of $T(\Graph)$ in~\eqref{eq:tors-graph-variat} paves the way to the application of surgery methods, similar to what has been systematically done for the {ground-state energy} since~\cite{Nic87}; for instance we can easily derive the following useful surgery principles.

\begin{proposition}\label{prop:surgery-0}
Suppose $\Graph$ is a  metric graph  satisfying Assumption~\ref{ass:graph}. Then the following assertions hold.
\begin{enumerate}[(1)]
\item {[Gluing vertices]} If $\NewGraph$ is obtained from $\Graph$ by gluing two different vertices \(\mv\neq \mw\) in ${\mVN}$, then $T(\NewGraph;{{\mVD}})\le T(\Graph;{{\mVD}})$ with equality if and only if the torsion \(v\) on \(\Graph\) satisfies \(v(\mv)=v(\mw)\).
\item {[Imposing Dirichlet conditions on additional vertices]} If 
$\mv\not\in {{\mVD}}$, then $T(\Graph;{{\mVD}}\cup\{\mv\})< T(\Graph;{{\mVD}})$.
\item {[Attaching pendant graphs with no Dirichlet vertices]} If $\NewGraph$ is formed by attaching a pendant graph $\NewGraph \setminus \Graph$ at only one vertex $\mv\in {\mVN}$ of $\Graph$, and if no Dirichlet conditions are imposed on $\NewGraph \setminus \Graph$, then $T(\Graph) \le T(\NewGraph)$.
\item {[Adding edges between points where the torsion function attains the same value]} If $\NewGraph$ is obtained from $\Graph$ by adding an edge $\me=\mv'\mv''$ between two vertices $\mv',\mv''$ of $\Graph$, then $T(\Graph)\le T(\NewGraph)$ provided the torsion function of $\Graph$ attains the same value on both $\mv',\mv''$.
\item {[Lengthening edges]} If $\NewGraph$ is obtained from $\Graph$ by lengthening one edge, then $T(\Graph) < T(\NewGraph)$.
\item {[Scaling the whole graph]} If $ \NewGraph $ is obtained from $ \Graph $ by scaling each edge with the factor $c>0 $, then 
$T (\Graph) = c^{-3} T (\NewGraph)$.
\item {[Unfolding parallel edges]} Let $\me_1,\me_2$ be two parallel edges with common endpoints $\mv_1,\mv_2\in \mV$. If $\NewGraph$ is obtained from $\Graph$  by replacing $\me_1,\me_2$ with a single edge $\me_0$ of length $\ell_{\me_0} = \ell_{\me_1}+\ell_{\me_2}$, then $T(\NewGraph)> T(\Graph)$.
\end{enumerate}
\end{proposition}

These surgery principles are taken over from~\cite[Lemma~2.3]{KenKurMal16} and \cite[Theorem~3.18.(2)]{BerKenKur19}, where they were developed to prove similar relations for {the eigenvalues} of modified metric graphs. The proof of the items in Proposition~\ref{prop:surgery-0} is very similar to that of their {eigenvalue} counterparts, up to replacing the role of {the respective eigenfunctions} by the torsion, but we prefer to formulate it for the sake of self-containedness.

\begin{proof}
(1)  For the inequality, we only need to observe that $H^1_0(\NewGraph)$ may be identified with the subspace (of codimension 1) that consists of functions in $H^1_0(\Graph)$ that coincide in \(\mv\) and \(\mw\). For the characterization of equality observe first that \(v(\mv)=v(\mw)\) yields that \(v\) is an admissible test function in the variational characterization of \(T(\NewGraph;{{\mVD}})\) which yields \(T(\NewGraph;{{\mVD}})\geq T(\Graph;{{\mVD}})\). If, on the other hand, \(T(\NewGraph;{{\mVD}})\geq T(\Graph;{{\mVD}})\) holds, then the torsion \(\tilde v\) on \(\NewGraph\) considered as an element of $H^1_0(\Graph)$  maximizes the Poly\'a quotient on \(\Graph\). Therefore, \(\tilde v\) has to be a scalar multiple of \(v\) which implies \(v(\mv)=v(\mw)\).

(2) Likewise, here we observe that $H^1_0(\Graph;{{\mVD}}\cup\{\mv\})$ is a subspace of $H^1_0(\Graph;{{\mVD}})$. Moreover, the inequality has to be strict, since the torsion on \(\Graph\) with respect to the smaller Dirichlet vertex set \({{\mVD}}\) has to be strictly positive in \(\mv\) by Corollary \ref{cor:positivity-of-torsion} and equality would yield that the torsion \(\tilde v\) with respect to the larger Dirichlet vertex set \({{\mVD}}\cup \{\mv\}\) has to be a scalar multiple of \(v\) in contradiction to \(\tilde v(\mv)=0\).

(3) Let $v$ be the torsion function of $\Graph$ and hence the maximizer of the Pólya quotient, and suppose that we obtain $ \NewGraph$ by attaching a pendant graph to a vertex $ \mv \in \Graph$.
 We extend $v$ to a function $\tilde v\in H_0^1(\NewGraph)$ by setting $\tilde v\equiv v(\mv)$ on $\NewGraph \setminus \Graph $. Then $v$ is a valid test function for $T(\NewGraph)$, whose Pólya quotient is not smaller than $T(\Graph)$. 

(4) Like in the proof of (3), extending on the new edge $\me\equiv \mv'\mv''$ the torsion function $v$ of $ \Graph $ by a constant equal to the common value $ v(\mv') =
v (\mv'') $ yields a test function for $ \NewGraph$ whose Pólya quotient is not smaller than $T(\Graph)$. 

(5) This is a direct consequence of Proposition~\ref{prop:hada}.

(6) The proof follows essentially (5). One needs to take into account that all edges are scaled with the same factor.

(7)
Let $v$ be the torsion function of $\Graph$. Suppose without loss of generality that $v_{|\me_1\cup\me_2}$ reaches its maximum $M>0$ at some point in $\me_1$, say $x_0 \in [0,\ell_{\me_1}]$. We define a test function $\tilde{v}$ on $\NewGraph$ by $\tilde{v} = v$ on $\NewGraph \setminus \me_0$, and, identifying $\me_0$ with the interval $[0,\ell_{\me_0}]$,
\begin{displaymath}
	\tilde{v} (x) := \begin{cases} v|_{\me_1} (x), \qquad & x \in [0,x_0],\\
	M, & x \in (x_0,x_0+\ell_{\me_2}),\\
	v|_{\me_1}(x+\ell_{\me_2}),\qquad & x \in [x_0+\ell_{\me_2},\ell_{\me_1}].
	\end{cases}
\end{displaymath}
One sees that $\tilde{v} \in H^1_0 (\NewGraph;{{\mVD}})$, hence it  is a valid test function for the Pólya quotient. Additionally, we see that
\[
\begin{aligned}
 \int_{\NewGraph} |\tilde{v}'|^2\ud x &	\le \int_{\mathcal G} |v'|^2\ud x\quad\hbox{and}\quad 
	\int_{\NewGraph} |\tilde{v}|\ud x &\ge \int_{\mathcal G} |v|\ud x,
	\end{aligned}
\]
from which $T(\tilde{\Graph};{{\mVD}})>T(\Graph;{{\mVD}})$ follows.

The strictness of the inequality follows by noting that in each case we have constructed a test function $\tilde{v}$ that cannot be a scalar multiple of the torsion,
since $\tilde{v}$ is equal to a nonzero constant on a subgraph of positive measure. Since the torsion function is the only admissible test function in $H^1_0(\NewGraph;{{\mVD}})$ whose Pólya quotient amounts to $T(\NewGraph;{{\mVD}})$, this proves the assertion.
\end{proof}


\begin{remark}
Proposition~\ref{prop:surgery-0}.(1) can arguably solve a quantum version of Braess paradox. In its original form in~\cite{Bra68}, the latter states that the efficiency of a traffic network may counterintuitively \emph{decrease} upon opening a new road; over the decades, it has been theoretically demonstrated and observed in action in many different fields. In the context of diffusion on quantum graphs, it can be reasonably formulated as follows:
\begin{quote}
The heat kernel may be strictly increasing upon increasing the connectivity of a metric graph $\Graph$ by gluing two points.
\end{quote}
This is due to the fact that the heat kernel on $\NewGraph$ (the metric graph arising by gluing two points of $\Graph$) is not dominated by the heat kernel on $\Graph$; indeed, the former is not even \emph{eventually} dominated by the latter, as observed in~\cite{GluMug21}. So, at any time $t$ there are two points $x_t,y_t$ in the metric graph that observe a decrease in diffusion features: in probabilistic terms, the odds of reaching $x_t$ from $y_t$ within $t$ decrease in the new graph configuration. This is perhaps unexpected, given that by \cite[Thm.~3.4]{BerKenKur19} the $k$-th eigenvalue of $\NewGraph$ is not larger than the $k$-th eigenvalue of $\Graph$, for $k=2$ (so convergence to equilibrium is not slower in the new configuration) and in fact for any further $k$, too. What Proposition~\ref{prop:surgery-0}.(1) shows is that, when integrated in time and space as in~\eqref{eq:intheacon}, these pathological effects are indeed negligible.
\end{remark}
{
As a consequence of Proposition~\ref{prop:surgery-0} we mention the following lemma, which can be proved like~\cite[Lemma~5.4]{KenKurMal16}. It will be used to derive an estimate for the torsional rigidity in terms of the \emph{inradius} $\Inr(\Graph;{{\mVD}})$ of \(\Graph\) with respect to \(|{{\mVD}}|\), i.e., 
	\[\Inr(\Graph):=\Inr(\Graph;{{\mVD}}):=\sup\{\dist(x;{{\mVD}})~|~x\in\mathcal G\}\]
where
	\[\dist(x;{{\mVD}}):=\min_{\mv\in{{\mVD}}}d_{\mathcal G}(x,\mv)\]
and \(d_{\mathcal G}\) denotes the path metric on \(\mathcal G\).
}
\begin{lemma}
\label{lem:reduction}
Given any compact, connected, non-empty metric graph $\Graph$, there exists a pumpkin chain $\NewGraph$ such that the following holds:
\begin{enumerate}
\item Exactly one of the outside vertices of \(\NewGraph\) is a Dirichlet vertex. The degree of said Dirichlet vertex is equal to $\Diri(\Graph)$.
\item $\Inr(\NewGraph) = \Inr(\Graph)$, $|\NewGraph| \leq |\Graph|$ and $|{\mVN}(\NewGraph)| \leq |{\mVN}(\Graph)|+1$;
\item $T(\NewGraph) \leq T(\Graph)$.
\end{enumerate}
\end{lemma}
\begin{proof}
	After gluing all vertices in \({{\mVD}}(\mathcal G)\) we may assume that there is exactly one Dirichlet vertex \(\mv_0\) in \(\mathcal G\) whose degree is \(\Diri(\Graph)\). Since \(\mathcal G\) is compact there exists some \(x\in\mathcal G\) so that \(\Inr(\Graph;{{\mVD}})=d_\mathcal G(\mv_0,x)\). Then, following the algorithm used in the proof of~\cite[Lemma~5.4]{KenKurMal16} and using Proposition \ref{prop:surgery-0}, one can construct a pumpkin chain as stated.
\end{proof}

{We refer to~\cite[Definition~5.3]{KenKurMal16} for a formal introduction of \textit{pumpkin chains}, which can be visualized as a concatenation of pumpkin graphs, see Figure~\ref{fig:pumpkin-chain}.

\begin{figure}[H]
\begin{tikzpicture}[scale=0.6]
\coordinate (a) at (0,0);
\coordinate (b) at (2,0);
\coordinate (c) at (6,0);
\coordinate (d) at (7,0);
\coordinate (e) at (10,0);
\draw (a) -- (b);
\draw[bend left]  (a) edge (b);
\draw[bend right]  (a) edge (b);
\draw (b) -- (c);
\draw[bend right=90]  (b) edge (c);
\draw[bend left]  (b) edge (c);
\draw (c) -- (d);
\draw[bend left]  (e) edge (d);
\draw[bend right]  (e) edge (d);
\draw[fill] (a) circle (2pt);
\draw[fill] (b) circle (2pt);
\draw[fill] (c) circle (2pt);
\draw[fill] (d) circle (2pt);
\draw[fill] (e) circle (2pt);
\end{tikzpicture}
\caption{A pumpkin chain}
\label{fig:pumpkin-chain}
\end{figure}
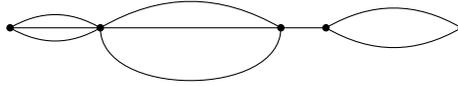}

We are finally in the position to prove our main lower estimate on $T(\Graph)$ in terms of the inradius of $\Graph$.

\begin{proposition}\label{prop:estim-low-inr}
Suppose $\Graph$ is a  metric graph  satisfying Assumption~\ref{ass:graph}. Then the torsional rigidity {with respect to an arbitrary Dirichlet vertex set \(\emptyset\neq{{\mVD}}\subset\mV\)} admits the lower bound
	\begin{equation}\label{eq:estimate-inradius-vertices}
		T(\Graph;{{\mVD}})\geq \frac{\Inr(\mathcal G;{{\mVD}})^3}{3(|\mV\setminus {{\mVD}}|+1)^3}.
	\end{equation}
\end{proposition}
\begin{proof}
{For the sake of notational simplicity, let us denote by  $|{\mVN}|$ the number of non-Dirichlet vertices.}
	By Lemma \ref{lem:reduction} it is sufficient to prove the estimate
		\[T(\Graph;\{\mv_0\})\geq \frac{\Inr(\mathcal G;{{\mVD}})^3}{3|{\mVN}|^3}\]
	for any pumpkin chain \(\mathcal G\) where {Dirichlet conditions are imposed on exactly one extremal vertex \(\mv_0\)}. For such a pumpkin chain, beginning from the Dirichlet vertex, let \(\mv_0,\mv_1,\ldots,\mv_m\) denote the vertices of \(\mathcal G\) where \(m=|{\mVN}|\) and let \(\ell_j\) denote the length of the edges connecting \(\mv_{j-1}\) and \(\mv_j\). Then
	\[\Inr(\mathcal G;\mV_{D})=d_\mathcal G(\mv_0,\mv_m)=\sum_{j=1}^m\ell_j\]
	holds and thus, by the pigeonhole principle, there exists some \(k\) with \(\ell_k\geq \frac{\Inr(\Graph;{{\mVD}})}{m}\). We now consider the following test function \(u\) on \(\mathcal G\): on an edge \(\me\simeq [0,\ell_j]\) connecting \(\mv_{j-1}\) and \(\mv_j\) we set \(u(x)=0\) if \(1\leq j <k\), \(u(x)=\frac{1}{2}x(2\ell_j-x)\) if \(j=k\), and \(u(x)=\frac{\ell_k^2}{2} \) if \(k< j\leq m\). Then, it is immediate to check that the Pólya quotient of \(u\) is bounded from below by \(\frac{\Inr(\mathcal G;{{\mVD}})^3}{3|{\mVN}|^3}\). Therefore, using \eqref{eq:torsion-var} yields the asserted estimate.
\end{proof}

If $\Graph$ is a tree, then the previous estimate can be sharpened.

\begin{proposition}
Let $\Graph$ be a  metric graph  satisfying Assumption~\ref{ass:graph}. Let, additionally, $\Graph$ be a tree. 

If $\mG$ has precisely one Dirichlet vertex, then
	\begin{equation}\label{eq:estimate-inradius-vertices-tree}
		T(\Graph)\geq \frac{1}{3}\Inr(\mathcal G;{{\mVD}})^3.
	\end{equation}

If $\mG$ has precisely two Dirichlet vertices $\mv,\mw$,  then
	\begin{equation}\label{eq:estimate-inradius-vertices-2}
		T(\Graph)\geq \frac{1}{12}\dist(\mv,\mw)^3.
	\end{equation}
\end{proposition}

\begin{proof}
Since $\mG$ is a tree, we can by Proposition~\ref{prop:surgery-0} recursively prune it -- only lowering its torsion function --, ending up with a path graph whose length is $\Inr(\mathcal G;{{\mVD}})$ or $\dist(\mv,\mw)$, respectively. Now the claim follows by Example~\ref{exa:basic-tors}.
\end{proof}

We have seen in Example~\ref{exa:basic-tors} that the torsional rigidity of open subsets of $\R$ can be explicitly computed. An interval is, of course, just a very elementary metric graph. Owing to Proposition~\ref{prop:surgery-0} we can extend~\eqref{eq:tors-1d} and~\eqref{eq:tors-1d-DN} to an estimate valid for all metric graphs: the following is the metric graph counterpart of the celebrated \emph{Saint-Venant inequality} for domains.

\begin{theorem}\label{thm:polya}
For any metric graph $\Graph$ one has
\begin{equation}\label{eq:tors-first-1}
T(\Graph)\le \frac{|\Graph|^3}{3};
\end{equation}
the inequality is actually an equality if and only if $\Graph$ is an interval with mixed Dirichlet/Neumann conditions.

If, additionally, $\Graph$ is doubly connected upon gluing all vertices in \({{\mVD}}\), then we have the improved estimate
\begin{equation}\label{eq:tors-first-2}
T(\Graph)\le \frac{|\Graph|^3}{12};
\end{equation}
in this case, the inequality is actually an equality if and only if $\Graph$ is a caterpillar graph.
\end{theorem}

We recall that a \textit{caterpillar graph} is a 2-regular pumpkin chain where one of the two endpoints (of degree two) is equipped with Dirichlet conditions, see Figure~\ref{fig:caterpillar}. Note that the torsional rigidity of a caterpillar graph of length \(|\Graph|\) is indeed \(\frac{|\Graph|^3}{12}\). This follows from Proposition \ref{prop:surgery-0}(1), \eqref{eq:tors-1d}, and the fact that the torsion on an interval with Dirichlet conditions imposed at both end points is symmetric with respect to the center point of the interval.
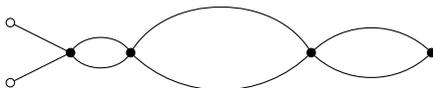
\begin{figure}[H]
\begin{tikzpicture}[scale=0.8]
\coordinate (a1) at (1,.5);
\coordinate (a2) at (1,-.5);
\coordinate (b) at (2,0);
\coordinate (c) at (3,0);
\coordinate (d) at (6,0);
\coordinate (e) at (8,0);
\draw (a1) -- (b);
\draw (a2) -- (b);
\draw[bend left=60] (b) edge (c);
\draw[bend right=60] (b) edge (c);
\draw[bend left=60] (c) edge (d);
\draw[bend right=45] (c) edge (d);
\draw[bend left=45] (d) edge (e);
\draw[bend right=45] (d) edge (e);
\draw[black,fill=white] (a1) circle (2pt);
\draw[black,fill=white] (a2) circle (2pt);
\draw[fill] (b) circle (2pt);
\draw[fill] (c) circle (2pt);
\draw[fill] (d) circle (2pt);
\draw[fill] (e) circle (2pt);
\end{tikzpicture}
\caption{A caterpillar graph: we recall that, as far as spectral or torsional properties are concerned, gluing Dirichlet vertices makes no difference.}
\label{fig:caterpillar}
\end{figure}

In the case of planar domains, this result was proved in \cite[Section~5]{Pol48} using a method based on Steiner symmetrization. We, too, will use a rearrangement method, but in this case we will rather follow an approach first proposed in~\cite{Fri05} and closer in spirit to Schwarz symmetrization.

\begin{proof}
First of all, take the torsion function $v$ and observe that each value $t$ between its minimum $0$ and its maximum $\|v\|_\infty$ is attained $n(t)$ times along $\Graph$, where
\begin{itemize}
\item $n(t)\ge 1$, since $v$ is continuous on the metric space $\Graph$;
\item $n(t)\ge 2$ for all $t\in [0,\|v\|_\infty]\setminus v(\mV)$ if $\Graph$ is doubly connected, as one sees adapting (with very minor modifications) the proof of \cite[Lemma~3.7]{BerKenKur17}.
\end{itemize}
Now, define the symmetrized version $v^*$ of the torsion function $v$ as in the proof of~\cite[Lemma~3]{Fri05}: letting $m_v(t)$ be the measure of the superlevel set $\{x\in \Graph:v(x)>t\}$ we can uniquely define a new function $v^*\in H^1(0,|\Graph|)$ in such a way that $v^*(0)=v(0)$, $v^*$ is monotonically increasing, and $m_v(t)=m_{v^*}(t)$ for a.e.\ $t\in [0,\|v\|_\infty]$
. Then one checks as in the proof of \cite[Theorem~2.1]{BanLev17} that 
\[
\int_{\Graph} |v'(x)|^2\ud x\ge \essinf\limits_{t\in [0,\|v\|_\infty]}n(t)^2 \int_0^{|\Graph|}|(v^*)'(y)|^2\ud y,
\]
see~\cite[equation (3.13)]{BanLev17}. On the other hand, by Cavalieri's principle 
\[
\int_{\Graph} |v(x)|^p\ud x=\int_0^{|\Graph|}|v^*(y)|^p\ud y
\]
for all $p\ge 1$. It follows in particular that 
\[
\frac{\|v\|^2_{L^1(\Graph)}}{\|v'\|^2_{L^2(\Graph)}}\le \frac{1}{\essinf\limits_{t\in [0,\|v\|_\infty]}n(t)^2}\frac{\|v^*\|^2_{L^1(0,|\Graph|)}}{\|(v^*)'\|^2_{L^2(0,|\Graph|)}};
\]
because $v^*$ is an admissible test function for the Pólya quotient of the interval $(0,|\Graph|)$ with mixed Dirichlet/Neumann conditions, this yields the asserted estimates \eqref{eq:tors-first-1} and \eqref{eq:tors-first-2}. Now as in the proof of \cite[Lemma~4.3]{BerKenKur17} equality in \eqref{eq:tors-first-1} yields \(n(t)=1\) for all \(t\in [0,\|v\|_\infty]\) while equality in \eqref{eq:tors-first-2} yields \(n(t)=2\) for all \(t\in [0,\|v\|_\infty]\setminus\psi(\mV)\) from which we infer the claimed characterizations on equality in \eqref{eq:tors-first-1} and \eqref{eq:tors-first-2} respectively.
\end{proof}

\begin{remark}\label{rem:polya-uplow}
(1) Theorem~\ref{thm:polya} is the metric graph counterpart of a celebrated isoperimetric inequality, which for domains was conjectured by Saint-Venant and eventually proved in~\cite{Pol48}. Let us remark that Makai provided in~\cite{Mak66} a different proof, which in turn yields an estimate in terms of volume and inradius: in our case, a hypothetical Makai-type inequality would read \( T(\Graph)<4|\Graph|\Inr(\Graph;{{\mVD}})^2\), but we have not been able to prove it.

(2) The \emph{upper} estimate in Theorem~\ref{thm:polya} does not have a \emph{lower} counterpart. To see this, let $\Graph_n$ be the graph consisting of $n$ disjoint intervals of length $n^{-1}$, each with Dirichlet conditions at both endpoints. Then
\[
|\Graph_n|\equiv 1\quad\hbox{but}\quad T(\Graph_n)=\sum_{k=1}^n \frac{|\Graph_n|^3}{12}=\frac{1}{12 n^2}\searrow 0.
\]

(3) Let us stress that~\eqref{eq:tors-first-1} can be regarded as a Faber--Krahn-type inequality: it states that
\[
|J_0|^3 T(J_0)^{-1}\le |\Graph|^3 T(\Graph)^{-1},
\]
where $J_0$ is the path graph with mixed Dirichlet/Neumann conditions;  equality holds if and only if $\Graph=J_0$ (up to isomorphism). Likewise, if $\Graph$ is doubly connected after gluing all vertices in ${{\mVD}}$, then by~\eqref{eq:tors-first-2}
\[
|J_1|^3 T(J_1)^{-1}\le |\Graph|^3 T(\Graph)^{-1},
\]
where if $J_1$ is the path graph with both Dirichlet conditions;
 equality if and only if $\Graph=J_1$ (up to isomorphism).
 
\end{remark}

The variational formulation \eqref{eq:tors-graph-variat} makes it possible to derive lower estimates on $T(\Graph)$: for instance, the distance function 
\[
\Graph\ni x\mapsto \dist (x;{{\mVD}})\in [0,\infty)
\]
is an admissible test function for the Pólya quotient. In fact, an even simpler test function is given by
\[
\psi:=\sum_{\me\in\mE}v_\me \1_\me,
\]
where $\1_\me$ is the characteristic function of the edge $\me$ and $v_\me$ is the torsion function of the individual edge $\me\equiv (0,\ell_\me)$ (with Dirichlet conditions imposed at all endpoints), i.e.,
\[
v_\me(x):=\frac{1}{2}x(\ell_\me-x),\qquad x\in (0,\ell_\me),\ \me\in\mE:
\]
the torsional rigidity of $\Graph$ is not smaller than the Pólya quotient evaluated at $\psi$, i.e.,
\begin{equation}\label{eq:firstlower}
T(\Graph)\ge \frac{1}{12}\sum_{\me\in \mE} \ell_\me^3
\end{equation}
we know from Theorem~\ref{thm:discrete-and-metric-torsion} that the accuracy of this estimate is measured by the discrete torsion function $T(\mG)$.

By Jensen's inequality, we deduce from~\eqref{eq:firstlower} that
\begin{equation}\label{eq:firstlower-2}
T(\Graph)\ge \frac{|\Graph|^3}{12 |\mE|^2}:
\end{equation}
this estimate becomes an equality if and only if $\Graph$ {is an equilateral flower with Dirichlet conditions imposed in the (only) vertex (equivalently, if $\Graph$} is 
a disconnected metric graph consisting of {equilateral} path graphs with Dirichlet conditions  imposed at all endpoints): {this is arguably the metric graph counterpart of Makai's inequality
\[
T(\Omega)> \frac{4|\Omega|^3}{3|\partial \Omega|^2}
\]
for convex planar domains $\Omega$, see~\cite{Mak59,Pol61}.}

The estimate \eqref{eq:firstlower} can indeed be improved.
In Example~\ref{exa:coll-tors-comput} we have computed the torsional rigidity of  \emph{stowers}: this class of graphs will prove useful when proving the following lower bound.

\begin{proposition}\label{thm:polya-nic}
There holds
\[
\frac{1}{12}\sum_{\me\in\mE}\ell_\me^3+\frac{1}{4}\left(\sum_{\me\in {\mED}}\ell_\me+2\sum_{\me\in \mE\setminus {\mED}}\ell_\me\right)^2\cdot \left(\sum_{\me\in {\mED}}\frac{1}{\ell_\me}\right)^{-1}
\le T(\Graph),
\]
where ${\mED}$ is the set of edges incident to vertices in ${{\mVD}}$; the inequality becomes an equality if $\Graph$ is a stower graph.
\end{proposition}

\begin{proof}
The proof is similar to that of \cite[Theorem~4.2]{KenKurMal16}. First of all, we glue all vertices in ${\mVN}$: by Proposition~\ref{prop:surgery-0}, this can only reduce the torsional rigidity. In this way we obtain a stower whose set of loop edges is $\mE\setminus {\mED}$ and whose set of leaf edges is ${\mED}$. The estimate now immediately follows from~\eqref{eq:stower-tors}.
\end{proof}

\section{Spectral estimates}\label{sec:spectralest}

In this section we are going to slightly change our viewpoint and study upper and lower bounds on products of (suitable powers of) the torsional rigidity and the ground-state energy in terms of metric quantities. Of course, this ansatz can still be understood as a kind of shape optimization assignment, but it additionally paves the road to the use of torsional rigidity as a spectral quantity in its own right, much like the Cheeger constant.

\subsection{Upper estimates: Pólya--Szeg\H{o} Inequality}
We being by observing that the product $\lambda_1(\mathcal G)$ admits an upper bound only depending on the total length of $\Graph$ and its torsional rigidity $T(\mathcal G)$.

\begin{proposition}\label{prop:polya-type}
For any subset \(\emptyset\neq{{\mVD}}\subset {\mV}\), there holds
	\begin{equation}\label{eq:polya-upper}
\lambda_1(\mathcal G;{{\mVD}})T(\mathcal G;{{\mVD}})< |\mathcal G|.
	\end{equation}
\end{proposition}

The corresponding estimate is well-known in the case of open domains $\Omega\subset \R^d$: indeed, the proof in~\cite[Section~5.4]{PolSze51} carries over verbatim to the case of metric graphs, yielding~\eqref{eq:polya-upper}. We write it down for the sake of self-containedness.

\begin{proof}
	Let $v$ be the torsion function  on $\mathcal G$. Because $v$ is also a valid test function for the Rayleigh quotient, the stated inequality follows using the Cauchy--Schwarz inequality and the variational characterization of \(\lambda_1(\mathcal G)\) from
\[
T(\Graph)= \frac{\|v\|^2_{L^1}}{\|v'\|^2_{L^2}} < \frac{|\Graph|\|v\|^2_{L^2}}{\|v'\|^2_{L^2}}\leq \frac{|\mathcal G|}{\lambda_1(\mathcal G)}.
\]
Indeed, the first inequality is strict, since equality in the Cauchy--Schwarz inequality
\[
\|v\cdot \1\|_{L^1}\le \|v\|_{L^2} \|\1\|_{L^2}
\]
holds precisely if $v'$ is a multiple of $\1$: this is not possible, since otherwise $v''=0$, against the definition of torsion.
\end{proof}

{
\begin{remark}
(1) Combining Proposition~\ref{prop:polya-type} and \cite[Theorem~4.4.3]{Plu21b} we get the upper bound
\[
T(\Graph;\mVD)<\Inr(\Graph;\mVD) |\Graph|^2 .
\]
This is partially reminiscent of the inequality
\[
T(\Omega)\le \kappa \Inr(\Omega)^2 |\Omega|,
\]
in terms of the inradius $ \Inr(\Omega)$ and volume $|\Omega|$,
proved by Makai for simply connected ($\kappa=4$, with strict inequality) or convex ($\kappa=\frac{4}{3}$) domains $\Omega\subset \R^2$~\cite{Mak62,Mak66}; see also~\cite[Section~6]{Bra22} for generalizations and historical remarks.

(2) Also, combining Proposition~\ref{prop:polya-type} and Proposition~\ref{prop:estim-low-inr} we obtain the upper estimate
\begin{equation}\label{eq:estimate-inradius-vertices-upper}
\lambda_1(\mathcal G;\mVD)<
 \frac{3|\Graph|(|\mV\setminus {{\mVD}}|+1)^3}{\Inr(\mathcal G;\mVD)^3}.
	\end{equation}
A lower estimate on \(\lambda_1(\mathcal G;\mVD)\) in terms of inradius for a special class of simply connected metric graphs (viz, metric trees with a center point) has been recently observed in~\cite{DufKenMug22}.
\end{remark}}

Let us elaborate on this simple method: in analogy with the case of manifolds with non-empty boundary~\cite{Che70,Yau75}, and following~\cite[Section~6]{Pos09b}, in the case of quantum graphs with non-empty Dirichlet vertex set ${{\mVD}}$ the \emph{Cheeger constant} we define the \emph{Cheeger constant} $h(\Graph;{{\mVD}})$ of $\Graph$ as
\[
h(\Graph):=h(\Graph;{{\mVD}}):=\inf \frac{|\partial \tilde{\Graph}|}{|\tilde{\Graph}|}.
\]
where the infimum is taken over all non-empty subgraphs $\tilde{\Graph}$ with ${\tilde{\Graph}}\cap {{\mVD}}=\emptyset$,  and $|\partial \tilde{\Graph}|$ is the number of points in the boundary of $\tilde{\Graph}$. (A different definition for graphs without Dirichlet vertices appears in~\cite{Nic87}, but is not appropriate for our purposes.) {A similar but more general result for domains is already known, see~\cite[Theorem 2]{BueErc11}.}

\begin{proposition}\label{prop:cheeger}
There holds
\begin{equation}\label{eq:polya-cheeger}
h^2(\Graph;{{\mVD}})T(\Graph;{{\mVD}})< |\Graph|.
\end{equation}
\end{proposition}

\begin{proof}
Again, let $v$ be the torsion function  on $\mathcal G$ and observe that, in particular, $v$ is a function of bounded variation vanishing in the Dirichlet vertices of the quantum graph. Again, the stated inequality follows using the Cauchy--Schwarz inequality and the variational characterization of the Cheeger constant in~\cite[Theorem~6.1]{DelRos16} from
\[
T(\Graph)= \frac{\|v\|^2_{L^1}}{\|v'\|^2_{L^2}} < \frac{|\Graph|\|v\|^2_{L^1}}{\|v'\|^2_{L^1}}\le 
 \frac{|\mathcal G|}{h^2(\mathcal G)}.
\]
Again, the first inequality is strict, since equality in the Cauchy--Schwarz inequality
\[
\|v'\cdot \1\|_{L^1}\le \|v'\|_{L^2} \|\1\|_{L^2}
\]
holds precisely if $v'$ is a multiple of $1$: this is not possible, since the torsion function  vanishes on ${{\mVD}}$.
\end{proof}

%

Different but comparable upper bounds for {$d$-dimensional} domains or manifolds can be found in~\cite[Corollary~2.5]{GioSmi10} and~\cite[Theorem~1.10]{CadGalLoi15}.

\begin{example}
(1) It is unclear how sharp the estimate~\eqref{eq:polya-upper} is. 
For stars on $k$ edges with natural conditions in the center and Dirichlet conditions in the leaves (and in particular for intervals with mixed Dirichlet/Neumann conditions, $k=1$; or with both Dirichlet endpoints, $k=2$) we indeed have in view of Example~\ref{exa:basic-tors}.(3)
\[
\lambda_1(\mathcal G)T(\mathcal G)=\frac{\pi^2k^2}{4|{\mathcal S}_k|^2}\frac{|{\mathcal S}_k|^3}{3k^2}=\frac{\pi^2}{12}|{\mathcal S}_k|\approx 0.8225\cdot |{\mathcal S}_k|,\qquad k=1,2,\ldots.
\]

{
In the case of domains $\Omega\subset \R^d$, $d\ge 2$,
the improved estimate
\begin{equation}\label{eq:notopt}
\frac{\lambda_1(\Omega)T(\Omega)}{|\Omega|}\le 1-\frac{2d\omega_d^\frac{2}{d}}{d+2}\frac{T(\Omega)}{|\Omega|^{1+\frac{2}{d}}}
\end{equation}
was proved in~\cite[Theorem~1.1]{BerFerNit16}, where $\omega_d$ is the measure of the ball of unit radius in $\R^d$. (However, again for $d\ge 2$, by~\cite[Thm.~1.2]{BerFerNit16} the quantity $\frac{T(\Omega)\lambda_1(\Omega)}{|\Omega|}$ is known \emph{not} to have a maximizer.) 

For $d=1$,  the formal metric graph counterpart
\[
\frac{\lambdaG T(\Graph)}{|\Graph|}\le 1-\frac{8}{3}\frac{T(\Graph)}{|\Graph|^{3}}\quad\hbox{is not true}
\]
 in general, as one sees already considering the case of path graphs with either mixed or purely Dirichlet boundary conditions.
}

(2) Also the estimate~\eqref{eq:polya-cheeger} seems to be far from sharp. Again for the class of metric graphs considered in (1) (i.e., equilateral stars on $k$ edges with natural conditions in the center and Dirichlet conditions in the leaves), we can reason as in~\cite[Example~6.3]{DelRos16} and immediately see that $h(\Graph)=\frac{k}{|\Graph|}$: in view of Example~\ref{exa:basic-tors}.(3) we conclude that 
\[
h^2({\mathcal S}_k)T({\mathcal S}_k)=\frac{k^2}{|{\mathcal S}_k|^2}\frac{|{\mathcal S}_k|^3}{3k^2}=\frac{1}{3}|{\mathcal S}_k|,\qquad k=1,2,\ldots.
\]

(3) {If, again, $\Graph$ is equilateral, then combining Propositions~\ref{thm:polya-nic} and~\ref{prop:cheeger}
we deduce
\begin{equation}\label{eq:cheeger2}
\begin{split}
h^2(\Graph;{{\mVD}})&<\frac{|\Graph|}{T(\Graph)}\\
&\le\frac{4|\mE|^2}{|\Graph|^2} \frac{3|{\mED}|}{|{\mED}|+3(|{\mED}|+2|\mE\setminus {\mED}|)^2},
\end{split}
\end{equation}
this can be compared with the bound $h(\Graph)\le \frac{2|\mE|}{|\Graph|}$ observed in \cite{KenMug16}  for \emph{possibly non-equilateral} quantum graphs \emph{without} any Dirichlet vertices.}
\end{example}



\subsection{Lower estimates: Kohler-Jobin Inequality}

{
Like in~\cite{FilMay12} it can be proved that for any eigenfunction $\varphi$ with associated eigenvalue $\lambda$ of $\Delta_{\Graph}$ the bound
\begin{equation}\label{eq:landscape-abstr-0}
|\varphi(x)|\le |\lambda|\|\varphi\|_\infty v(x)\qquad\hbox{for all }x\in \Graph
\end{equation}
holds. In Section \ref{sec:torsland}, we provide a general framework to obtain this estimate {and, in particular, a slightly sharper version, see Proposition~\ref{eq:stein-0}}.
}

A simple lower estimate on the {ground-state energy} $\lambda_1(\Graph)$ based on the torsion function can be easily deduced from~\eqref{eq:landscape-abstr-0}, taking the supremum of both sides, {see also Corollary~\ref{cor:abstr-giosmi} below.}


\begin{corollary}
There holds
\begin{equation}\label{eq:giosmispec-app}
1 \le \lambdaGVD  \|v\|_\infty.
\end{equation}
\end{corollary}

We have seen in Example~\ref{exa:basic-tors}.(1) that in the case of path graphs with one or two Dirichlet endpoints, the torsion function  $v$ and the lowest Laplacian eigenvalue $\lambda_1$ satisfy
\(
\lambda_1 \|v\|_\infty=\frac{\pi^2}{8}\simeq 1.2337
\). 
For domains, it was conjectured in \cite{Ber12} that \(\frac{\pi^2}{8}\) is the optimal lower bound for the product \(\lambda_1\|v\|_\infty\).
\begin{example}\label{exa:giosmi}
If $\Graph$ is an equilateral star on $k$ edges, then the estimate~\eqref{eq:giosmispec-app}  yields, in view of Example~\ref{exa:basic-tors}.(3),
\[
\lambdaG \ge \frac{2k^2}{|\Graph|^2}:
\]
this can be compared with 
\begin{equation}\label{eq:nicaisebkkm}
\lambdaG \ge \frac{\pi^2}{|\Graph|^2}\quad\hbox{and}\quad \lambdaG \ge \frac{k}{|\Graph|^2},
\end{equation}
cf.~\cite[Lemma~4.2]{BerKenKur17} and~\cite[Theorem~4.4.3]{Plu21b}, which are worse for all $k\ge 3$.
\end{example}

{
\begin{remark}
In the case where $(\lambda,\varphi)$ is an eigenpair of a Schrödinger operator with a random potential, the inequality~\eqref{eq:landscape-abstr-0} is the starting point in Filoche's and Mayboroda's theory of \emph{landscape functions}: it was interpreted in~\cite{FilMay12} as a means to control the localization of the eigenfunctions $\varphi$ by means of the ``landscape'' of the torsion function $v$: i.e., $\varphi$ may be large only in regions where so is $v$, and in fact it was numerically observed that  the regions of localizations for eigenfunctions 
 correspond very well, at low energies, with the regions enclosed by the ``valleys'' of $v$. {The role of \(v\) as a landscape function in a more abstract setting will be further discussed in Section \ref{sec:torsland}.}
 
  If one tries to carry over these ideas to our quantum graph setting, the valleys of the torsion function $v$ are simply given by the discrete set of local minima of $v$.	Unfortunately, by Lemma~\ref{lem:no-local-min} there exist no local minima of the torsion function outside ${{\mVD}}$: this seems to underpin the belief that the scope of the usage of torsion function as a landscape function for metric graphs is rather limited, as long as natural conditions are imposed outside ${{\mVD}}$; {in the last few years, the role of the ``effective potential'' $\frac{1}{v}$ has been emphasized, instead,  see~\cite{ArnDavFil19,
HarMal20}.} 
\end{remark}
}

This section is devoted to prove a lower estimate of different flavor.

A breakthrough in the torsional theory of higher dimensional bodies came in the 1970s, when Kohler-Jobin could finally prove a sharp bound on the product of the torsional rigidity with the square of the {ground-state energy}: this bound, first conjectured by Pólya and Szeg\H{o} in~\cite{PolSze51}, fully established the torsional rigidity as a relevant quantity in spectral geometry and shape optimization. This section is devoted to derive a Kohler-Jobin-type estimate in the context of metric graphs.

While $\frac{\lambdaG  T(\Graph)}{|\Graph|}=\frac{\pi^2}{12}$ for path graphs with either mixed or purely Dirichlet boundary conditions, \cite[Remark~2.4]{BerButVel15} suggests that the positive quantity $\frac{\lambdaG  T(\Graph)}{|\Graph|}$ can be arbitrarily small, as a sequence $(\Graph_n)_{n\in\N}$ of metric graphs (with each $\Graph_n\setminus{{\mVD}}$ disconnected) displaying this behavior can be written down explicitly. 

It turns out that the trouble with the product $\lambdaG T(\Graph)$ is simply a wrong scaling: as suggested by Proposition~\ref{prop:surgery-0}.(6), we rather have to study
\[
\lambdaG T(\Graph)^\frac{2}{3}.
\]
While Remark~\ref{rem:polya-uplow}.(2) shows that a lower bound on this quantity cannot simply be obtained as a straightforward corollary of the classical isoperimetric inequality for the {ground-state energy} of quantum graphs in \cite[Théorème 3.1]{Nic87}, in the case of domains the celebrated Kohler-Jobin Inequality does state that ``balls minimize the {ground-state energy} among sets with given torsional rigidity''. Conjectured in~\cite{PolSze51}, it was proved in~\cite{Koh78} and extended in~\cite{BerIve13,Bra14} to rougher domains and nonlinear operators.

{
The following  Kohler-Jobin-type inequality for connected, compact, finite quantum graphs is the main result of this section. Unlike in the case of domains, there are \textit{two} isoperimetric inequalities for the {ground-state energy} of $\Delta_{\Graph}$ with ${{\mVD}}=\emptyset$: one for the general (connected) case (\cite[Théorème~3.1]{Nic87}) and one for the \textit{doubly} connected case (\cite[Lemma~4.3]{BerKenKur17}). We have seen in Theorem~\ref{thm:polya} that the same is true of the torsional rigidity: in either case, the optimizers are intervals (connected case) and caterpillar graphs (doubly connected case). The same phenomenon appears in the following, too.

\begin{theorem}\label{thm:kj-brasco-qg}
We have
\begin{equation}\label{eq:kj-vanilla}
\left(\frac{\pi}{\sqrt[3]{24}}\right)^2\le \lambdaGVD T(\Graph;{{\mVD}})^\frac{2}{3},
\end{equation}
with equality if and only if $\Graph$ is a path graph with mixed Dirichlet/Neumann boundary conditions.

If, additionally, $\Graph$ is doubly connected after gluing all Dirichlet vertices, then
\[
\left(\frac{\pi}{\sqrt[3]{12}}\right)^2\le \lambdaGVD T(\Graph;{{\mVD}})^\frac{2}{3},
\]
with equality if and only if $\Graph$ is a caterpillar graph, see Figure~\eqref{fig:caterpillar}. 
\end{theorem}

The proof is rather technical and will be subdivided in several Lemmata. We will follow Kohler-Jobin's construction~\cite[Section~1.3.1]{Koh78}  in the modified and generalized version proposed by Brasco in~\cite{Bra14}, thus eluding technical problem like the failure of analyticity of eigenfunctions on metric graphs.
}

{
Throughout this section, let $\psi$ be a positive eigenfunction for the lowest eigenvalue of $-\Delta_\Graph$ on the metric graph $\Graph$ with Dirichlet conditions at ${{\mVD}}$. Moreover, given a function \(u\in H_0^1(\Graph;{{\mVD}})\) with \(u\geq 0\) on \(\Graph\), we define for $t\in [0,\|u\|_\infty]$ the level and superlevel sets
\[
\{u=t\}:=\{x\in\Graph: u(x)=t\}\quad\hbox{and}\quad \{u>t\}:=\{x\in \Graph: u(x)>t\}
\]
as well as
\[
\gamma_u(t):=\sum_{x\in \{u=t\}}|u'(x)|\quad \hbox{and}\quad
\alpha_u(t):=|\{u>t\}|.
\]
As in \cite{Bra14}, we use the following notion of \emph{reference functions}.
\begin{definition}
	A function \(u\in H_0^1(\Graph;{{\mVD}})\) with \(u\geq 0\) on \(\Graph\) is called a \emph{reference function} if the set $\{u=t\}$ is finite for a.e.\ $t\in (0,\|u\|_\infty)$, and additionally
		\[t\mapsto \frac{\alpha_u(t)}{\gamma_u(t)}\in L^\infty(0,\|u\|_\infty).\]
\end{definition}
To begin with, we are going to prove that $\psi\in H^1_0(\Graph;{{\mVD}})$ is indeed a reference function.
Following~\cite[Lemma~3.2, Remark~3.3]{Bra14}, let us observe the following.
\begin{lemma}\label{lem:brasco0}
Let $\psi\in H^1_0(\Graph;{{\mVD}})$ be a positive eigenfunction for the lowest eigenvalue $\lambda_1(\Graph;{{\mVD}})$ of $-\Delta_\Graph$. Then
\[
\sum_{x\in \Grapht}|\psi'(x)|=\lambda_1(\Graph;{{\mVD}})\int_{\Gammat}\psi(x)\ud x\qquad\hbox{for a.e.\ }t\in [0,\|\psi\|_\infty].
\]
\end{lemma}

\begin{proof}
For all $s\in [0,\|\psi\|_\infty]$ the coarea formula yields immediately
\[
\int_{\Gammas}|\psi'(x)|^2\ud x=\int_s^{\|\psi\|_\infty}\sum_{x\in \Graphtau}|\psi'(x)|\ud \tau
=\int_s^{\|\psi\|_\infty}\gamma_\psi(\tau)\ud \tau.
\]
Also, we see that 
\[
\begin{aligned}
\int_{\Gammas}|\psi'(x)|^2\ud x&=\int_{\Graph}\psi'(x)\left((\psi-s)_+\right)'(x)\ud x\\
&=-\int_{\Graph}\Delta_{\Graph} \psi(x)(\psi-s)_+(x)\ud x\\
&=\lambda_1(\Graph;{{\mVD}})\int_{\Graph}\psi(x)(\psi-s)_+(x)\ud x:
\end{aligned}
\]
in particular, for a.e.\ $t\in [0,\|\psi\|_\infty]$ and small $h$ we have
\[
\begin{aligned}
\int_t^{t+h}\gamma_\psi(s)\ud s&=
\int_t^{\|\psi\|_\infty}\gamma_\psi(t)\ud t-\int_{t+h}^{\|\psi\|_\infty}\gamma_\psi(t)\ud s\\
&=\int_{\Gammat}|\psi'(x)|^2\ud x-\int_{\Graph(=t+h)}|\psi'(x)|^2\ud x\\
&=\lambda_1(\Graph;{{\mVD}})\int_{\Graph}\psi(x)\left((\psi-t)_+(x)-(\psi-t-h)_+(x)\right)\ud x.
\end{aligned}
\]
Dividing by $h$ and passing to the limit $h\to 0$ we find
\[
\begin{aligned}
\gamma_\psi(t)&=\lambda_1(\Graph;{{\mVD}})\int_{\Graph}\psi(x)\left((\psi-t)_+\right)'(x)\ud x\\
&=\lambda_1(\Graph;{{\mVD}})\int_{\Gammat}\psi(x)\ud x,
\end{aligned}
\]
as we wanted to prove.
\end{proof}

\begin{corollary}\label{cor:brasco1}
The first eigenfunction $\psi\in H^1_0(\Graph;{{\mVD}})$ for the lowest eigenvalue $\lambda_1(\Graph;{{\mVD}})$ of $-\Delta_\Graph$ is a reference function.
\end{corollary}

\begin{proof}
To begin with, observe that
\[
 \int_{\Gammat} \psi(x)\ud x\ge  t \alpha_\psi(t),
\]
whence by Lemma~\ref{lem:brasco0}
\[
0\le \frac{\alpha_\psi(t)}{\gamma_\psi(t)}=\frac{\alpha_\psi(t)}{\lambda \int_{\Gammat} \psi(x)\ud x}\le \frac{1}{\lambda t}.
\]
In particular, this shows that 
\[
\frac{\alpha_\psi(t)}{\gamma_\psi(t)}\le \frac{2}{\lambda\|\psi\|_\infty}\qquad \hbox{for }t\ge \frac{1}{2}\|\psi\|_\infty.
\]
If however $t<\frac{1}{2}\|\psi\|_\infty$, then 
\[
\int_{\Gammat}\psi(x)\ud x\ge \int_{\{\psi=\frac{1}{2}\|\psi\|_\infty\}}\psi(x)\ud x
\ge \frac{1}{2}\|\psi\|_\infty \left|\alpha_\psi\left(\frac{1}{2}\|\psi\|_\infty\right)\right|>0
\]
and, again by Lemma~\ref{lem:brasco0},
\[
\frac{\alpha_\psi(t)}{\gamma_\psi(t)}\le \frac{2|\Graph|}{\lambda\|\psi\|_\infty}\qquad \hbox{for }t\ge \frac{1}{2}\|\psi\|_\infty.
\]
This concludes the proof.
\end{proof}

We follow Kohler-Jobin\footnote{Strictly speaking, Kohler-Jobin restricts in~\cite{Koh78} to a class $\mathcal C$ that is actually smaller, as it consists of continuously differentiable functions only. In our context this raises a few technical issues, due to the fact that neither the torsion function  nor the eigenfunctions of $\Graph$ are continuously differentiable on the metric space $\Graph$, as long as $\Graph$ contains vertices of degree larger than 2. Brasco observed that defining $\mathcal C$ by means of Lipschitz continuity is sufficient to make sense of the integrals  in $\Tmod$. This relaxation will be crucial for our study.}
 and introduce the following.
 
\begin{definition}
Let $u$ be a reference function. Then the \emph{modified torsional rigidity with respect to $u$} is
\begin{equation}\label{eq:tmod-def}
\begin{split}
\Tmod(\Graph,u)&:=
\sup_{g\in \mathcal C} \left(2\int_{\Graph} (g\circ u)(x)\ud x- \int_{\Graph}|(g\circ u)'(x)|^2\ud x\right),
\end{split}
\end{equation}
where
\[
\mathcal C:=\{g\in W^{1,\infty}(0,\|u\|_\infty):g (0)=0\}.
\]
\end{definition}
We will mostly take $u$ to be $\psi$, a positive eigenfunction associated with the lowest positive eigenvalue $\lambda_1$; in this case, we will simply write
\begin{equation}\label{eq:tpsi}
\Tmod(\Graph):=\Tmod(\Graph,\psi).
\end{equation}
Observe that 
\begin{equation}\label{eq:estimate-tormod-and-tor}
	\Tmod(\Graph,u)\le T(\Graph),
\end{equation}
holds for any reference function \(u\) as $(\mathcal C\circ u)\subset H^1_0(\Graph)$. The crucial observation by Kohler-Jobin, which carries over to the Brasco's generalized modified torsional rigidity, is that the supremum in $\Tmod(\Graph,u)$ is actually attained, see Lemma~\ref{lem:bra2}. Before proving it, we observe that the quantity in~\eqref{eq:tmod-def} does not change upon taking the supremum on the even smaller class
\[
\tilde{\mathcal C}:=\{g\in W^{1,\infty}(0,\|u\|_\infty):g (0)=0,\ g'\ge 0\},
\]
i.e., \(\tilde{\mathcal C}\) is the convex cone of non-decreasing functions \(g\in W^{1,\infty}(0,\|u\|_\infty)\) with \(g(0)=0\).

\begin{lemma}\label{lem:bra2-e}
For any reference function \(u\), there holds
\[
\Tmod(\Graph,u)=
\sup_{g\in \tilde{\mathcal C}} \left(2\int_{\Graph} (g\circ u)(x)\ud x- \int_{\Graph}|(g\circ u)'(x)|^2\ud x\right).
\]
\end{lemma}

\begin{proof}
Clearly, it suffices to prove that 
\[
\sup_{g\in \mathcal C} \left(2\int_{\Graph} (g\circ u)(x)\ud x- \int_{\Graph}|(g\circ u)'(x)|^2\ud x\right)\le \sup_{\tilde{g}\in \tilde{\mathcal C}} \left(2\int_{\Graph} (\tilde{g}\circ u)(x)\ud x- \int_{\Graph}|(\tilde{g}\circ u)'(x)|^2\ud x\right).
\]
Indeed, for all $g\in \mathcal C$ we can consider the function \(\tilde g\) given by
\begin{equation}\label{eq:def-g-tilde}
\tilde{g}(t):=\int_0^t |g'(\tau)|\ud \tau;
\end{equation}
because $\tilde{g}'=|g'|$, $\tilde{g}\in \tilde{\mathcal C}$. Now,
\begin{equation}\label{eq:estimate-tormodfunctional-g-gtilde}
2\int_{\Graph} (g\circ u)(x)\ud x- \int_{\Graph}|(g\circ u)'(x)|^2\ud x\le 2\int_{\Graph} (\tilde{g}\circ u)(x)\ud x- \int_{\Graph}|(\tilde{g}\circ u)'(x)|^2\ud x
\end{equation}
as $g\le \tilde{g}$.
\end{proof}

\begin{remark}
	Suppose \(g\in \mathcal C\setminus \tilde{\mathcal C}\), i.e., there is a measurable subset \(A\subset (0,\|\psi\|_\infty)\) with positive measure, so that \(g'<0\) holds almost everywhere on \(A\). Thus, there exists some \(t_0\in(0,\|\psi\|_\infty)\) with \(\tilde{g}(t)>g(t)\) for all \(t>t_0\) where \(\tilde{g}\) is the function defined in \eqref{eq:def-g-tilde}. It follows that the inequality in \eqref{eq:estimate-tormodfunctional-g-gtilde} is strict if \(g\in \mathcal C\setminus \tilde{\mathcal C}\), and therefore a maximizer of \eqref{eq:tmod-def} has to be an element of \(\tilde{\mathcal C}\). The following lemma proves the existence and uniqueness of such a maximizer and characterizes said maximizer.
\end{remark}
\begin{lemma}\label{lem:bra2}
For any  reference function \(u\), let $g_0$ be the function defined by
\[
	g_0(t):=\int_0^{t} \frac{\alpha_u(s)}{\gamma_u(s)}\ud s, \quad t\in (0,\|u\|_\infty).
\]
Then $g_0\in {\mathcal C}$ and
\[
\|g_0\|_1=\Tmod(\Graph,u)=\int_0^{\|\psi\|_\infty} \frac{\alpha_u(t)^2}{\gamma_u(t)}\ud t;
\] 
indeed, $g_0$ is the unique maximizer of \eqref{eq:tmod-def}.
\end{lemma}

\begin{proof}
We follow \cite[Proposition~3.5]{Bra14}. For all $g\in \tilde{\mathcal C}$ we find
\begin{align*}
\left(2\int_{\Graph} (g\circ u)(x)\ud x- \int_{\Graph}|(g\circ u)'(x)|^2\ud x\right) & = 2\int_0^{\|u\|_\infty}\left( g'(t)\cdot|\{u>t\}|- \frac{g'(t)^2}{2}\sum_{x\in\{u=t\}}|u'(x)|\right)\ud t\\
 & = 2\int_0^{\|u\|_\infty}\left( g'(t)\alpha_u(t)- \frac{g'(t)^2}{2}\gamma_u(t)\right)\ud t
\end{align*}
by Cavalieri's principle (left addend) and the coarea formula (right addend). Now, one immediately checks that the following the pointwise estimate
\[
	 g'(t)\alpha_u(t)- \frac{g'(t)^2}{2}\gamma_u(t)
	\leq \max_{s\in\mathbb R}\left( s\alpha_u(t)- \frac{s^2}{2}\gamma_u(t)\right) = \frac{1}{2}\frac{\alpha_u(t)^2}{\gamma_u(t)}
\]
holds for almost every \(t\in (0,\|u\|_\infty)\) with equality if and only if $g'(t)=\frac{\alpha_u(t)}{\gamma_u(t)}$. We obtain
\[
\left(2\int_{\Graph} (g\circ u)(x)\ud x- \int_{\Graph}|(g\circ u)'(x)|^2\ud x\right)\le \int_0^{\|u\|_\infty}\frac{\alpha_u(t)^2}{\gamma_u(t)}\ud t,
\]
with equality if and only if \(g=g_0\), hence \(g_0\) necessarily defines the unique maximizer of~\eqref{eq:tmod-def} and
\begin{equation}\label{eq:tmod-def-attain}
\begin{split}
\Tmod(\Graph,u)&=\int_0^{\|u\|_\infty}\frac{\alpha_u(t)^2}{\gamma_u(t)}\ud t,
\end{split}
\end{equation}
and it is immediate to check that \(\|g_0\|_{L^1}=\Tmod(\Graph,u)\). This concludes the proof.
\end{proof}
}
We can now derive, as in~\cite[Proposition~3.8]{Bra14}, a fundamental isoperimetric inequality.
\begin{lemma}\label{lem:bra3}
Let \(u\in H^1(\Graph,{{\mVD}})\) be any reference function. If $J_0$ is a path graph with mixed Dirichlet/Neumann conditions with
\[
T(J_0)=\Tmod(\Graph,u),
\]
then $|J_0|\le |\Graph|$, with equality if and only if $\Graph=J_0$.

If, additionally, $\Graph$ is doubly connected after gluing all Dirichlet vertices, and if $J_1$ is an interval with Dirichlet conditions at both endpoints with
\[
T(J_1)=\Tmod(\Graph,u),
\]
then $|J_1|\le |\Graph|$, with equality if and only if $\Graph$ is a caterpillar graph.
\end{lemma}


\begin{proof}
By the Faber--Krahn-type inequality for torsional rigidity in Remark~\ref{rem:polya-uplow}.(3) we know that
\[
1\le \frac{T(\Graph)}{T(J_0)}\le \frac{|\Graph|^3}{|J_0|^3 };
\]
and likewise
\[
1\le \frac{T(\Graph)}{T(J_1)}\le \frac{|\Graph|^3}{|J_1|^3 },
\]
in the doubly connected case, upon  gluing all vertices in ${{\mVD}}$.

If moreover $|J_0|=|\Graph|$ in the general case (resp., $|J_1|=|\Graph|$ in the doubly connected case), then $T(J_0)=T(\Graph;{{\mVD}})$ (resp., $T(J_1)=T(\Graph;{{\mVD}})$)  and we know from Theorem~\ref{thm:polya} that this implies that $J_0=\Graph$ (resp., $\Graph$ is caterpillar graph).
\end{proof}

Let us now complete the last step before attacking the proof of the Kohler-Jobin Inequality: to this aim, we will need to construct the symmetrization $\psi^*$ of the the eigenfunction $\psi$. We see next how to perform this construction, which differs from the classical Schwarz symmetrization. 

\begin{lemma}\label{lem:bra4}
There exists a monotonically decreasing function ${\psi^*}\in H^1\left(0,\sqrt[3]{3\Tmod(\Graph)}\right)$ with ${\psi^*}(\sqrt[3]{3\Tmod(\Graph)})=0$, {where $\Tmod(\Graph)$ is defined as in~\eqref{eq:tpsi},} such that
\begin{equation}\label{eq:kj-rearr-0}
\int_0^{\sqrt[3]{3\Tmod(\Graph)}} |{{\psi^*}}'(x)|^2\ud x=\int_{\Graph} |\psi'(x)|^2\ud x\quad \hbox{and}\quad 
\int_0^{\sqrt[3]{3\Tmod(\Graph)}} |\psi ^*(x)|^2\ud x\ge \int_{\Graph} |\psi (x)|^2 \ud x.
\end{equation}
The second inequality in~\eqref{eq:kj-rearr-0} is actually an equality if and only if $\Graph$ is a path graph of length $\sqrt[3]{3\Tmod(\Graph)}$.

Let, additionally, $\Graph$ be doubly connected after gluing all Dirichlet vertices. 

Then there exists a function ${\psi^*}\in H^1_0\left(-\frac{1}{2}\sqrt[3]{12\Tmod(\Graph)},\frac{1}{2}\sqrt[3]{12\Tmod(\Graph)}\right)$ that is
\begin{itemize}
\item  monotonically decreasing on $[0,\frac{1}{2}\sqrt[3]{12\Tmod(\Graph)}]$, 
\item symmetric about $0$, and
\item such that
\begin{equation}\label{eq:kj-rearr}
\int_0^{\sqrt[3]{12\Tmod(\Graph)}} |{{\psi^*}}'(x)|^2\ud x=\int_{\Graph} |\psi'(x)|^2\ud x\quad \hbox{and}\quad 
\int_0^{\sqrt[3]{12\Tmod(\Graph)}} |\psi ^*(x)|^2\ud x\ge \int_{\Graph} |\psi (x)|^2 \ud x.
\end{equation}
\end{itemize}
The second inequality in~\eqref{eq:kj-rearr} is an equality if and only if $\Graph$ is a caterpillar graph of length $\sqrt[3]{12\Tmod(\Graph)}$.
\end{lemma}
\begin{proof}
We now follow~\cite[Proposition~4.1]{Bra14}. The main idea of the proof is to consider the modified torsional rigidity of the set $\Gammat$, however with respect to the modified function $(\psi-t)_+$ restricted to the subgraph \(\Graph_t=\Gammat\) with Dirichlet conditions imposed in the boundary \(\partial \Graph_t\). The modified torsional rigidity  $\Tmod(\Graph_t,(\psi-t)_+)$ can be computed using the fact that
\begin{equation}\label{eq:alpha-psi-modified-function}
\alpha_{(\psi-t)_+}(s)=\alpha_\psi(t+s)\quad \text{and} \quad
\gamma_{(\psi-t)_+}(s)=\gamma_\psi(t+s)
\end{equation}
for \(s\in [0,\|\psi\|_\infty-t]\). {Corollary \ref{cor:brasco1} and \eqref{eq:alpha-psi-modified-function} yield that \((\psi-t)_+\) is a reference function on \(\Graph_t\).} Using Lemma \ref{lem:bra2}, we obtain
\[
T(t):=\Tmod(\Graph_t,(\psi-t)_+)=\int_0^{\|\psi\|_\infty-t} \frac{\alpha_{(\psi-t)_+}(s)^2}{\gamma_{(\psi-t)_+}(s)}\ud s=\int_t^{\|\psi\|_\infty} \frac{\alpha_\psi(\tau)^2}{\gamma_\psi(\tau)}\ud \tau;
\]
in particular $T\in W^{1,\infty}(0,\|\psi\|_\infty)$ and 
\[
T'(t)=-\frac{\alpha_\psi(t)^2}{\gamma_\psi(t)}<0\qquad \hbox{for a.e.\ }t\in [0,\|\psi\|_\infty].
\]
It follows that $T$ is invertible as a function between $[0,\|\psi\|_\infty]$ and its range $[0,{T_{\mathrm{mod}}(\Graph)}]$; we thus consider $\phi:=T^{-1}\in W^{1,\infty}$ and find, by the coarea formula and a change of variables,
\[
\int_{\Graph}|\psi'(x)|^2 \ud x=-\int_0^{T_{\mathrm{mod}}(\Graph)} \phi'(\tau)\sum_{x\in \{\psi=\phi(\tau)\}} |\psi'(x)| \ud \tau=-\int_0^{{T_{\mathrm{mod}}(\Graph)}} \phi'(\tau)\gamma_\psi(\phi(\tau)) \ud \tau
\]
and
\[
\phi'(\tau)=-\frac{\gamma_\psi(\phi(\tau))}{\alpha_\psi(\phi(\tau))^2}.
\]
We conclude that
\begin{equation}\label{eq:KJ-Dirichlet-Form-Symmetrized}
\int_{\Graph}|\psi'(x)|^2\ud x=\int_0^{{T_{\mathrm{mod}}(\Graph)}}\alpha_\psi(\phi(\tau))^2\phi'(\tau)^2\ud \tau.
\end{equation}
Kohler-Jobin's original idea was to rearrange the values of $\psi$ using the torsional rigidity of its superlevel sets $\Gammat$. In our context, the first step is to introduce the function $R$ that maps each $\tau\in [0,{T_{\mathrm{mod}}(\Graph)}]$ to the unique $R(\tau)\in \mathbb R$ such that 
\begin{equation}\label{eq:KJ-formula-for-R-doub}
	\tau\stackrel{!}{=}T\left([-R(\tau),R(\tau)]\right)=\frac{2 R(\tau)^3}{3}
\end{equation}
in the doubly connected case where we impose Dirichlet conditions at \(\pm R(t)\), or else
\begin{equation}\label{eq:KJ-formula-for-R-gen}
	\tau\stackrel{!}{=}T\left([0,R(\tau)]\right)=\frac{R(\tau)^3}{3}.
\end{equation}
in the general case where we impose Dirichlet conditions at \(R(t)\) and Neumann conditions at \(0\); i.e., the range of \(R(\tau)\) is \([0,\frac{1}{2}\sqrt[3]{12{T_{\mathrm{mod}}(\Graph)}}]\) in the doubly connected case and \([0,\sqrt[3]{3{T_{\mathrm{mod}}(\Graph)}}]\) in the general case.

We are finally in the position to introduce the Kohler-Jobin symmetrization of $\psi$: this is the function \(\psi^*\) with domain whose domain is \([-\frac{1}{2}\sqrt[3]{12{T_{\mathrm{mod}}(\Graph)}},\frac{1}{2}\sqrt[3]{12{T_{\mathrm{mod}}(\Graph)}}]\) in the doubly connected case and \([0,\sqrt[3]{3{T_{\mathrm{mod}}(\Graph)}}]\) in the general case and that is given by
\[
\psi^*(x)=\kappa(\tau)\quad\hbox{ whenever } |x|=R(\tau),
\]
where the function $\kappa$ will be chosen as the solution of 
\begin{equation}
	\left\{\begin{aligned}
		-\kappa'(\tau)\cdot 2R(\tau) & =-\phi'(\tau)\alpha_\psi(\phi(\tau))\\
		\kappa({T_{\mathrm{mod}}(\Graph)}) & =0	
	\end{aligned}\right.
\end{equation}
in the doubly connected case, and
\begin{equation}
	\left\{\begin{aligned}
		-\kappa'(\tau)\cdot R(\tau) & =-\phi'(\tau)\alpha_\psi(\phi(\tau))\\
		\kappa({T_{\mathrm{mod}}(\Graph)}) & =0	
	\end{aligned}\right.
\end{equation}
in the general case, respectively.
Then, by choice of \(\psi^*\) we have \(\alpha_{\psi^*}(\kappa(\tau))=2R(\tau)\) in the doubly connected case, and \(\alpha_{\psi^*}(\kappa(\tau))=R(\tau)\) in the general case, and we obtain
\begin{equation}\label{eq:compare-psi-and-kappa}
	-\kappa'(\tau)\cdot \alpha_{\psi^*}(\kappa(\tau))  =-\phi'(\tau)\alpha_\psi(\phi(\tau)).
\end{equation}
Now, to abbreviate the notation, we set \(J:=[-\frac{1}{2}\sqrt[3]{12{T_{\mathrm{mod}}(\Graph)}},\frac{1}{2}\sqrt[3]{12{T_{\mathrm{mod}}(\Graph)}}]\) in the doubly connected case, and \(J:=[0,\sqrt[3]{3{T_{\mathrm{mod}}(\Graph)}}]\) in the general case. Then, using \eqref{eq:KJ-formula-for-R-doub} and \eqref{eq:KJ-formula-for-R-gen} 
it can be shown that 
\begin{equation*}
	|(\psi^*)'(x)|=2\kappa'(\tau)R(\tau)^2=\frac{1}{2}\kappa'(\tau) \alpha_{\psi^*}(\kappa(\tau))^2
\end{equation*}
and
\begin{equation*}
		|(\psi^*)'(x)|=\kappa'(\tau)R(\tau)^2=\kappa'(\tau) \alpha_{\psi^*}(\kappa(\tau))^2
\end{equation*}
holds for all \(x\in J\) with \(|x|=R(\tau)\) in the doubly connected and general case, respectively. Then, using the coarea formula along with \eqref{eq:compare-psi-and-kappa} and \eqref{eq:KJ-Dirichlet-Form-Symmetrized} we obtain
\begin{align*}
	\int_J |{\psi^*}'(x)|^2\ud x & =\int_0^{T_{\mathrm{mod}}(\Graph)} \kappa'(\tau) \sum_{x\in \{\psi^*=\kappa(\tau)\}} |{\psi^*}'(x)|\ud \tau\\
	& = \int_0^{T_{\mathrm{mod}}(\Graph)}\kappa'(\tau)^2\alpha_{\psi^*}(\kappa(\tau))^2 \ud \tau\\
	& = \int_0^{{T_{\mathrm{mod}}(\Graph)}}\alpha_\psi(\phi(\tau))^2\phi'(\tau)^2\ud \tau\\
	& = \int_{\Graph}|\psi'(x)|^2\ud x.
\end{align*}
Now, note that for \(\tau\in[0,{T_{\mathrm{mod}}(\Graph)}]\) we have
	\[\tau = \Tmod(\Graph_{\phi(\tau)},(\psi-\phi(\tau))_+)= T(J\cap [-R(\tau),R(\tau)]).\]
Therefore, Lemma~\ref{lem:bra3} yields \(\alpha_{\psi}(\phi(\tau))\geq \alpha_{\psi^*}(\kappa(\tau))\). (In the doubly connected case, we use that \(\Graph_t\) is doubly connected after gluing all points in \(\partial \Graph_t\).) Using \eqref{eq:compare-psi-and-kappa}, we infer \(-\phi'(\tau)\leq  -\kappa'(\tau)\) which in turn yields -- after integrating from \(\tau\) to \({T_{\mathrm{mod}}(\Graph)}\) -- that \(\phi(\tau)\leq \kappa(\tau)\). Thus, Cavalieri's principle gives
\begin{align*}
	\int_{\Graph}|\psi(x)|^2\ud x & = 2\int_0^{\|\psi\|_\infty}t \alpha_\psi(t) \ud t \\
	& = -\int_0^{T_{\mathrm{mod}}(\Graph)} \phi(\tau) \phi'(\tau) \alpha_\psi(\phi(\tau)) \ud \tau \\
	& \leq -\int_0^{T_{\mathrm{mod}}(\Graph)} \kappa(\tau) \kappa'(\tau) \alpha_{\psi^*}(\kappa(\tau)) \ud \tau \\
	& = \int_J|\psi^*(x)|^2\ud x.
\end{align*}

To conclude the proof, we observe that equality in the previous inequality yields \(\alpha_\psi(\phi(t))=\alpha_{\psi^*}(\kappa(\tau))\) for all \(\tau\). By Lemma \ref{lem:bra3} this is only possible if \(\Graph_t\) is a caterpillar graph in the doubly connected case or a path graph in the general case.
\end{proof}

{
We are now finally in the position to complete the proof of Theorem~\ref{thm:kj-brasco-qg}.
}

\begin{proof}[Proof of Theorem~\ref{thm:kj-brasco-qg}]
We now adapt the proof of \cite[Theorem~1.1]{Bra14}, see \cite[Section~5]{Bra14}.

To begin with, we consider the interval 
\[
J_0=\left[0,\sqrt[3]{3\Tmod(\Graph)}\right]\ :
\]
 by Lemma~\ref{lem:bra4} we can consider the decreasing function ${\psi^*}\in H^1(J_0)$ with ${\psi^*}(\sqrt[3]{3\Tmod(\Graph)})=0$ that satisfies~\eqref{eq:kj-rearr-0}. Therefore,
\begin{equation}\label{eq:chain-inequalities}
\left(\frac{\pi}{\sqrt[3]{24}}\right)^2=\lambda_1({J_0})T(J_0)^\frac{2}{3}\le \frac{\|{{\psi^*}}'\|_2^2}{\|{{\psi^*}}\|_2^2}\Tmod(\Graph)^\frac{2}{3}\le \frac{\|{\psi}'\|_2^2}{\|{\psi}\|_2^2}\Tmod(\Graph)^\frac{2}{3}\le \lambdaG T(\Graph)^\frac{2}{3}
\end{equation}
{where we used the min-max principle and Lemma~\ref{lem:bra3} for the first inequality, Lemma~\ref{lem:bra4} for the second inequality, and the min-max principle and~\eqref{eq:estimate-tormod-and-tor} for the third inequality.}

Likewise, in the doubly connected case we consider the interval
	\[J_1=\left[-\frac{1}{2}\sqrt[3]{12\Tmod(\Graph)},\frac{1}{2}\sqrt[3]{12\Tmod(\Graph)}]\right],\]
impose Dirichlet conditions on both endpoints, and consider the radially symmetric function ${\psi^*}\in H^1_0\left(J_1\right)$ that satisfies~\eqref{eq:kj-rearr}.
Accordingly, {using the same arguments as in \eqref{eq:chain-inequalities},} we obtain
\[
\left(\frac{\pi}{\sqrt[3]{12}}\right)^2=\lambda_1({J_1})T(J_1)^\frac{2}{3}\le \frac{\|{{\psi^*}}'\|_2^2}{\|{{\psi^*}}\|_2^2}\Tmod(\Graph)^\frac{2}{3}\le \frac{\|{\psi}'\|_2^2}{\|{\psi}\|_2^2}\Tmod(\Graph)^\frac{2}{3}\le \lambdaG T(\Graph)^\frac{2}{3}.
\]
Equality in both chains of inequalities implies that 
\[
\frac{\|{{\psi^*}}'\|_2^2}{\|{{\psi^*}}\|_2^2}= \frac{\|{\psi}'\|_2^2}{\|{\psi}\|_2^2},
\]
thus the statement about equality follows from the corresponding statement in Lemma \ref{lem:bra4}.
\end{proof}

{
\section{Appendix: Torsion as landscape function}\label{sec:torsland}

Let us formulate an abstract version of a comparison principle between eigenfunctions and torsion function.
The following generalizes known bounds for Laplace or Schrödinger operators with Dirichlet boundary conditions on open domains of $\R^d$ in \cite[Theorem 1, Equation (0.5)]{BanCar01},
\cite[Lemma 1.1]{GioSmi10}, 
 \cite[Theorem 5, Equations (23) and (25)]{Ber12},
\cite{FilMay12}, and \cite[Theorem 2]{Ste17}.

\begin{proposition}\label{lem:landscape-abstr}
Let $A$ be a closed linear operator on $C_b(X)$, where $X$ is a locally compact metric space; or else on $L^p(X)$,  for some $p\in [1,\infty]$ and some $\sigma$-finite measure space $X$. Let $(\lambda,\varphi)$ be an eigenpair of $A$. 

Finally, let $\rho$ be a positive function in $C(X)$, resp.\  in $L^p(X)$, such that $\frac{\varphi}{\rho}$ is bounded, resp.\ essentially bounded. Then the following assertions hold.

\begin{enumerate}[(1)]
\item If \(A\) has positive inverse (in the sense of Banach lattices), then
\begin{equation}\label{eq:landscape-superabstr}
|\varphi(x)|\le |\lambda|\left\|\frac{\varphi}{\rho}\right\|_\infty A^{-1}\rho (x)\qquad\hbox{for all/a.e.\ }x\in X.
\end{equation}
\item If $-A$ generates a positive semigroup, then
\begin{equation}\label{eq:landscape-steiner}
|\varphi(x)|\le \left\|\frac{\varphi}{\rho}\right\|_\infty \e^{t\real \lambda}\e^{-tA} \rho(x)\qquad\hbox{for all/a.e.\ }x\in X\hbox{ and all }t\ge 0.
\end{equation}
\item If \(A\) has a positive inverse \emph{and} \(-A\) generates a positive semigroup, then
\begin{equation}\label{eq:landscape-steiner-doch-general}
|\varphi(x)|\le |\lambda|\left\|\frac{\varphi}{\rho}\right\|_\infty \e^{t\real \lambda}\e^{-tA} A^{-1}\rho(x)\qquad\hbox{for all/a.e.\ }x\in X\hbox{ and all }t\ge 0.
\end{equation}
\end{enumerate}
\end{proposition}

\begin{proof}
(1) We find for all/a.e.\ $x\in X$
\begin{equation}\label{eq:filmay-veryabs}
|\varphi(x)|=|\lambda A^{-1}\varphi(x)|\le |\lambda| A^{-1}\left(\left\|\frac{\varphi}{\rho}\right\|_\infty \rho\right)(x)=|\lambda|\left\|\frac{\varphi}{\rho}\right\|_\infty A^{-1}{\mathbf \rho}(x).
\end{equation}

(2) In order to prove \eqref{eq:landscape-steiner}, observe that by the Spectral Mapping Theorem $(\e^{-\lambda t},\varphi)$ is an eigenpair of $\e^{-tA}$ and, hence, due to positivity of $(\e^{\real \lambda t}\e^{-tA})_{t\ge 0}$,
\begin{equation}\label{eq:landscape-steiner-proof}
|\varphi(x)|=|\e^{\lambda t}\e^{-tA}\varphi(x)|\le \left\|\frac{\varphi}{\rho}\right\|_\infty  \e^{t\real \lambda}\e^{-tA} \rho(x)\qquad\hbox{for all/a.e.\ }x\in X.
\end{equation}
Taking the infimum over $t\ge 0$ concludes the proof of \eqref{eq:landscape-steiner}.

Finally, (3) is obtained by combining the arguments in (1) and (2).
\end{proof}

Recall that, if $A$ is invertible, then $A^{-1}$ is certainly positive whenever $-A$ generates a positive \textit{contraction} semigroup.

The following corollary is obtained by choosing  \(\rho=\mathbf 1\) and \(\rho=A^{-1}\mathbf 1\) in Proposition \ref{lem:landscape-abstr}. The assumptions on $X$ in the following can clearly be weakened: we omit the details for the sake of simplicity.

\begin{corollary}\label{cor:abstr-giosmi}
Let $A$ be a closed linear operator on $C(X)$, where $X$ is a compact metric space; or else on $L^p(X)$,  for some $p\in [1,\infty]$ and some finite measure space $X$. Let $A$ have positive inverse, and denote by $v$ the \emph{abstract torsion function}, i.e., the  (positive) solution of $Av={\mathbf 1}$. 

Then the following assertions hold for any eigenpair $(\lambda,\varphi)$.
	\begin{enumerate}[(1)]
	\item
If \(\varphi\) is bounded, then
\begin{equation}\label{eq:landscape-abstr}
|\varphi(x)|\le |\lambda|\|\varphi\|_\infty v(x)\qquad\hbox{for all/a.e.\ }x\in X.
\end{equation}
	\item If $-A$ generates a positive semigroup and \(\frac{\varphi}{v}\) is bounded, then
\begin{equation}\label{eq:landscape-steiner-maybe}
|\varphi(x)|\le \left\|\frac{\varphi}{v}\right\|_\infty  \e^{t\real \lambda}\e^{-tA} v(x)\qquad\hbox{for all/a.e.\ }x\in X \hbox{ and all }t\ge 0.
\end{equation}
\item If $-A$ generates a positive semigroup and \(\varphi\) is bounded, then
\begin{equation}\label{eq:landscape-steiner-doch}
|\varphi(x)|\le |\lambda|\left\|\varphi\right\|_\infty  \e^{t\real \lambda}\e^{-tA} v(x)\qquad\hbox{for all/a.e.\ }x\in X\hbox{ and all }t\ge 0.
\end{equation}
	\end{enumerate}
\end{corollary}

In particular, we deduce from \eqref{eq:landscape-abstr} and \eqref{eq:landscape-steiner-maybe} that
\begin{equation}\label{eq:giosmispec}
	1 \leq |\lambda|_{\min} \|{v}\|_\infty\quad \hbox{and}\quad 
		1 \leq |\lambda|_{\min} \e^{t\real \lambda_{\min}}\|\e^{-tA} v\|_\infty\hbox{ for all }t\ge 0,
\end{equation}
 if additionally \(v\) is bounded, where $|\lambda|_{\min}:=\min\{|\lambda|\ |\ \lambda\in\sigma_p(A)\}$: we have already encountered the first of these estimates in~\eqref{eq:giosmispec-app}.

\medskip
Observe that if $A$ admits compactly supported eigenfunctions -- this is, e.g., often the case for Laplacians on metric graphs -- then the proof of Proposition \ref{lem:landscape-abstr} even yields the sharper estimates
\begin{equation}\label{eq:even}
\begin{split}
|\varphi(x)|&\le |\lambda|\|\varphi\|_\infty A^{-1} {\mathbf 1}_{\supp \varphi}(x),\\
|\varphi(x)|&\le \left\|\varphi\right\|_\infty \e^{t\real \lambda}\e^{-tA} {\mathbf 1}_{\supp \varphi}(x),\\
|\varphi(x)|&\le \left\|\frac{\varphi}{v}\right\|_\infty \e^{t\real \lambda}\e^{-tA} A^{-1} {\mathbf 1}_{\supp \varphi}(x),\\
|\varphi(x)|&\le |\lambda|\|\varphi\|_\infty \e^{t\real \lambda}\e^{-tA} A^{-1} {\mathbf 1}_{\supp \varphi}(x),
\end{split}
\qquad\hbox{for all/a.e.\ }x\in X\hbox{ and all }t\ge 0.
\end{equation}

Indeed, the assumptions of Proposition~\ref{lem:landscape-abstr} are satisfied if we take $X=\Graph$ and $A=-\Delta_{\Graph}$: in fact, all eigenfunctions of $-\Delta_{\Graph}$ are bounded, since they belong to $H^1(\Graph)\hookrightarrow L^\infty(\Graph)$; furthermore, the quotient $\frac{\varphi}{v}$ is bounded for any eigenfunction $\varphi$, in view of the properties of the torsion function discussed in Section~\ref{sec:reduction} and the well-known fact that eigenfunctions of $-\Delta_{\Graph}$ behave like $o(x)$ as $x\to \mv\in\mVD$, see Example~\ref{exa:basic-tors}.(1).

Bearing in mind that $\e^{t\Delta_\Graph}v$ is for all $t\ge 0$ a continuous function, we hence obtain the following, a close analog of~\cite[Theorem~4.1 and Remark~4.1]{HarMal20}. 

\begin{proposition}\label{prop:steiner}
Let $-\Delta_\Graph \varphi=\lambda\varphi$, and let as usual $v$ denote the torsion function on a metric graph $\Graph$ that satisfies the Assumption~\ref{ass:graph}. Then
\begin{equation}\label{eq:stein-0}
|\varphi(x)|\le \lambda\left\|\varphi\right\|_\infty (-\Delta_\Graph^{-1} {\mathbf 1}_{\supp \varphi})(x)\le\lambda\left\|\varphi\right\|_\infty v(x) \qquad\hbox{for all }x\in \Graph.
\end{equation}
Furthermore,
\begin{equation}\label{eq:stein-1}
|\varphi(x)|\le \left\|\frac{\varphi}{v}\right\|_\infty \inf_{t\ge 0}\e^{\lambda t}\e^{t\Delta_\Graph} (-\Delta_\Graph^{-1} {\mathbf 1}_{\supp \varphi})(x)\le  \left\|\frac{\varphi}{v}\right\|_\infty \inf_{t\ge 0}\e^{\lambda t}\e^{t\Delta_\Graph} v(x)\qquad\hbox{for all }x\in \Graph
\end{equation}
as well as
\begin{equation}\label{eq:stein-2}
|\varphi(x)|\le \lambda \left\|\varphi\right\|_\infty \inf_{t\ge 0}\e^{\lambda t}\e^{t\Delta_\Graph} (-\Delta_\Graph^{-1} {\mathbf 1}_{\supp \varphi})(x)\le \lambda \left\|\varphi\right\|_\infty \inf_{t\ge 0}\e^{\lambda t}\e^{t\Delta_\Graph} v(x)\qquad\hbox{for all }x\in \Graph.
\end{equation}
\end{proposition}

Observe that Corollary~\ref{cor:abstr-giosmi} can be also applied to Laplacians on \textit{combinatorial} graphs (with Dirichlet conditions) or even more general $Z$-matrices; see also~\cite{FilMayTao21} for recent, more sophisticated localization results for such matrices.

\begin{remark}
(1) {The proofs of estimates \eqref{eq:landscape-abstr} and \eqref{eq:giosmispec} already available in the literature rely upon the additional 
assumption that $A$ is self-adjoint and/or that it has compact resolvent.} The estimate \eqref{eq:landscape-steiner-doch} was first obtained in \cite[Theorem~2]{Ste17}, whose proof is based on a Feynman--Kac-type formula that is assumed to hold with respect to the Brownian motion generated by the relevant Schrödinger operator.

(2) Also, \eqref{eq:landscape-abstr} yields
\begin{equation}\label{eq:giosmispec-eff}
	\frac{\|\varphi\|_1}{\|\varphi\|_\infty} \leq |\lambda| T(X)
\end{equation}
 for the \emph{abstract torsional rigidity} $T(X):=\|v\|_1$ and any eigenpair $(\lambda,\varphi)$ of $A$.
 
The  \emph{efficiency} $E(X):=\frac{\|\varphi_1\|_1}{|X|\|\varphi_1\|_\infty}$ is commonly studied for the ground state $\varphi_1$ of domains, see~\cite{BerDelDiB21} and references therein.

(3) Following a suggestion in~\cite[Section~4]{HarMal20}, Proposition~\ref{lem:landscape-abstr}.(1) can be generalized by observing that the inequality \eqref{eq:filmay-veryabs} is satisfied not only by the torsion function $v=A^{-1}{\mathbf 1}$, but by any -- possibly better behaved -- ``super-torsion function'', i.e., by any $v\ge A^{-1}{\mathbf 1}$ (in fact, by any $v\ge A^{-1}{\mathbf 1}_{\supp \varphi}$). However, it is not clear if this brings any advantages in our setting, given that a reasonably explicit formula for the torsion function is available on metric graphs.

(4) Localization for operators that do not satisfy a maximum principle is a popular topic, see~\cite{LefGonDub16} and references therein.
 Let us explicitly observe that if $A$ is merely invertible but its inverse is not positive, \eqref{eq:landscape-superabstr}--\eqref{eq:landscape-abstr}--\eqref{eq:landscape-steiner}--\eqref{eq:landscape-steiner-maybe} can still be replaced by corresponding estimates involving the terms $|A^{-1}|\mathbf 1(x)$, $ |A^{-1}|\mathbf 1(x)$, $|\e^{-tA}| \rho(x)$, $|\e^{-tA}| v(x)$, respectively, provided $A^{-1}$ and $\e^{-tA}$ have a modulus in the sense of~\cite[Section IV.1]{Sch74}, see also~\cite[Section~C-II]{Nag86}: for instance, on $C(X)$ or $L^p(X)$ as in Proposition~\ref{lem:landscape-abstr}, this is especially the case if $A^{-1}$, resp.\ $\e^{-tA}$, is a kernel operator, and in particular for general square matrices with complex coefficients.
%
These generalizations are straightforward and in line with what was already observed in~\cite{FilMay12}. 

Furthermore, let us mention that Proposition~\ref{lem:landscape-abstr}.(2)--(3) and its corollaries can be extended in a natural way to operators that generate semigroups which are merely \emph{individually eventually positive}: this is especially the case for distinguished realizations of higher order elliptic operators on bounded domains, see~\cite{DanGluKen16b}, and metric graphs, see~\cite{GreMug20}.
\end{remark}
}

\bibliographystyle{plain}

\end{document}